\documentclass{amsart}
\usepackage{amsmath,latexsym,amssymb,amsfonts, dsfont}

\newtheorem{thm}{Theorem}[section]
\newtheorem{defn}[thm]{Definition}
\newtheorem{lemma}[thm]{Lemma}
\newtheorem{remark}[thm]{Remark}
\newtheorem{remark/definition}[thm]{Remark/Definition}
\newtheorem{prop}[thm]{Proposition}
\newtheorem{cor}[thm]{Corollary}

\newcommand{\cA}{\mathcal{A}}

\newcommand{\cB}{\mathcal{B}}
\newcommand{\cD}{\mathcal{D}}

\newcommand{\Sb}{\Sigma_{\mathcal{B}}}
\newcommand{\Sbd}{\Sigma_{{\mathcal{B}:\mathcal{D}}}}

\newcommand{\id}{\text{Id}}
\newcommand{\cM}{\mathcal{M}}
\newcommand{\cfree}{\framebox{c}}

\newcommand{\CR}{{}^cR}

\newcommand{\lra}{\longrightarrow}

\newcommand{\X}{\mathcal{X}}

\newcommand{\bx}{\cB\langle \mathcal{X} \rangle}

\usepackage{amssymb}
\input xy
\xyoption{all}

\newcommand{\ncspace}[1]{\ensuremath{#1}_{\text{nc}}}

\newcommand{\ball}{\mathbb{B}}
\newcommand{\Nilp}{\text{Nilp}}
\newcommand{\ten}[1]{\mathbf{T}(#1)}

\newcommand{\cV}{\mathcal{V}}
\newcommand{\cW}{\mathcal{W}}

\title[infinite divisibility, non-commutative Bercovici-Pata bijection]{infinite divisibility and a non-commutative Boolean-to-free Bercovici-Pata bijection}
\author{S. T. Belinschi}\address{Department of Mathematics and Statistics, University of Saskatchewan,
106 Wiggins Road, Saskatoon, SK, S7N 5E6, CANADA, and
\newline
Institute of Mathematics ``Simion Stoilow'' of the Romanian Academy.
}
\email{belinschi@math.usask.ca}
\author{M. Popa}
\address{Center for Advanced Studies in Mathematics at the Ben Gurion  University of Negev, P.O. B. 653, Be'er Sheva 84105, Israel,
and
\newline Institute of Mathematics ``Simion Stoilow'' of the Romanian Academy, P.O. Box 1-764, Bucharest, RO-70700, Romania}
\email{popa@math.bgu.ac.il}
\author{V. Vinnikov}
\address{Department of Mathematics, Ben Gurion University of Negev, Be'er Sheva 84105, Israel}
\email{vinnikov@cs.bgu.ac.il}
\thanks{Research of STB was supported by a Discovery grant from NSERC Canada, and a University of Saskatchewan start-up grant.}

\date{\today}
\begin{document}


\begin{abstract}
We use the theory of fully matricial, or noncommutative, functions to investigate infinite divisibility and limit theorems in operator-valued noncommutative probability. Our main result is an operator-valued analogue for the Bercovici-Pata bijection. An important tool is Voiculescu's subordination property for operator-valued free convolution.
\end{abstract}

\maketitle

\section{introduction}
The study of problems in operator algebras from a probabilistic perspective recorded numerous successes in the recent decades. Notions of independence specific to noncommutative probability setting - especially Voiculescu's free independence - were shown to be important for encoding structural properties of certain operator algebras. Numerous analogues of classical probabilistic notions and phenomena were found to hold for free, Boolean or conditionally free independences. In particular, conditional expectations in the context of noncommutative algebras are old and well-known; however, noncommutative independences over subalgebras with respect to conditional expectations are new and not yet very well studied and understood: it was Voiculescu \cite{V*} who generalized his own free independence to freeness (with amalgamation) over a noncommutative subalgebra, while for the monotone independence of Muraki \cite{muraki}, Boolean independence of Speicher and Woroudi \cite{SW} and the conditionally free independence
  and convolution of Bo\.zejko, Leinert and Speicher \cite{bsk}, the generalization was achieved by Popa \cite{P,mv-bool} and \cite{mvcomstoc}.

Let us briefly introduce some of the more important notions
from noncommutative probability that we will study in this paper. We shall call a {\em noncommutative probability space} a pair $(\cA, \varphi)$, where $\mathcal A$ is a unital $\ast$-algebra over the complex numbers and $\varphi\colon\mathcal A\to\mathbb C$ is a positive functional normalized so that it carries the unit $1\in\mathcal A$ of $\mathcal A$ into the complex number 1. The algebra $\mathcal A$ plays the role of the algebra of complex-valued measurable functions on a classical probability space and $\varphi$ plays the role of integration with respect to the probability measure. By following this analogy, an {\em operator-valued noncommutative probability space} is naturally defined as a triple $(\mathcal A, E_\mathcal B,\mathcal B)$, where $\mathcal A$ is again a unital $*$-algebra, $\mathcal B$ is a $\ast$-subalgebra of $\mathcal A$ containing the unit of $\mathcal A$ and $E_\mathcal B\colon\mathcal A\to\mathcal B$ is a conditional expectation, i.e. a positive linear $\mathcal B-\mathcal B$ bimodule map. (In some contexts, this definition will turn out to be too restrictive or too broad; thus, we will usually specify our requirements, assumptions and notations on a case-by-case basis.) When $\mathcal B=\mathbb C$, we deal with an ordinary noncommutative probability space. Elements $X\in\mathcal A$ are called {\em random variables} or (in the second context) {\em operator-valued (or $\mathcal B$-valued) random variables}.

If $(\cA, \varphi)$ is a non-commutative probability space, the distribution of a self-adjoint element $X$ of $\cA$ (non-commutative random variable) is a real measure $\mu_X$ described via
  \[ \int t^n d\mu_X(t)=\varphi(X^n).\]
  As shown in \cite{V*}, for the operator-valued non-commutative random variables, the appropriate analogue of the distribution is an element of $\Sigma_\cB$, the set of positive conditional expectations from the $\ast$-algebra of non-commutative polynomials with coefficients in some C$^\ast$-algebra $\cB$ to $\cB$. In this setting, the op-valued analogues of compactly supported real measures are positive conditional expectations from $\Sigma_\cB$ with moments not growing faster than exponentially (see below).

In probability theory limit theorems play a central role. Historically, among the first results proved for each new type of noncommutative independence was a Central Limit Theorem: Voiculescu identified the Wigner law as the free central limit for scalar-valued free independence \cite{V-JFA} and its operator-valued analogue \cite{V*}, Muraki \cite{muraki} showed that the arcsine distribution is the monotone central limit and Popa \cite{P} identified an operator-valued analogue for the arcsine, Speicher and Woroudi \cite{SW} proved the Boolean central limit theorem, and Bo\.zejko, Leinert and Speicher \cite{bsk} described the pairs of measures that appear as conditionally free central limits. However, the ``most general'' limit theorems involve so-called infinitesimal arrays, and the limit distributions are usually shown to identify with {\em infinitely divisible} distributions. Some descriptions/characterizations of infinite divisibility are known for all noncommutative scalar independences \cite{BVIUMJ,SW,muraki,ADK,W}, but very little is known about operator-valued infinite divisibility; until recently, the only exception we know of was Speicher's work \cite{speicherhab}. In a recent breakthrough, Popa and Vinnikov \cite{popa-vinnikov} gave a description of free, Boolean and conditionally free infinitely divisible distributions in terms of their linearizing transforms that parallels the results of Bercovici and Voiculescu \cite{BVIUMJ}, Speicher and Woroudi \cite{SW} and Krysztek \cite{ADK}, respectively. In this paper we will use some of the results from \cite{popa-vinnikov}, the subordination result of Voiculescu \cite{V2} and the theory of noncommutative functions to prove a Hin\c{c}in type theorem for free and conditionally free convolutions and to identify operator-valued analogues of the Bercovici-Pata bijection \cite{BP}. In addition, we prove a generalization  to conditionally free convolution of the result of Belinschi and Nica \cite{BN} which states that the Boolean Bercovici-Pata bijection is a homomorphism with respect to free {\em multiplicative} convolution.

Our paper is organized as follows: in the second section we define the main notions and tools that we shall use, and prove a few preparatory results, in Section 3 we prove a Hin\v{c}in type theorem for certain infinitesimal arrays of operator-valued random variables; using this result, we obtain a restricted Boolean-to-free Bercovici-Pata bijection, in Section 4 we prove our main result, the Boolean-to-conditionally free Bercovici-Pata bijection, and finally in Section 5 we show that the Boolean-to-conditionally free Bercovici-Pata bijection for scalar-valued random variables is a homomorphism with respect to multiplicative c-free convolution.

\section{independences, transforms and subordination}

We start with a precise definition of an operator valued distribution. Generally, we will assume that $\mathcal B$ is a  unital C*-algebra.
We will denote by $\bx$  the $\ast$-algebra freely generated by $\cB$ and the selfadjoint symbol $\X$. We will also use the notation $\cB\langle \X_1, \X_2, \dots\rangle$ for the $\ast$-algebra freely generated by $\cB$ and the non-commutating self-adjoint symbols $\X_1, \X_2, \dots$. The set of all positive conditional expectations from $\bx$ to $\cB$ will be denoted by $\Sb$.

 For $\cB\subseteq\cD$ a unital inclusion of C$^\ast$-algebras, we denote by $\Sbd$ the set of all unital,  positive $\cB$-bimodule maps $\mu:\bx\lra\cD$ with the property that for all positive integers $n$ and all $\{f_i(\X)\}_{i=1}^n\subset \bx$ we have that:
\begin{equation}\label{prop}
\bigl[\mu(f_j(\X)^\ast f_i(\X))]_{i,j=1}^n\geq 0\ \text{in $M_n(\cB)$}.
\end{equation}
Remark that $\Sb=\Sigma_{\cB:\cB}$, as an easy consequence of Exercise 3.18 from \cite{paulsen}.

Using these notations, we define the distribution of an operator-valued random variable:
\begin{defn}\label{defn1}
  If $\cB\subseteq\cA$, $\cB\subseteq \cD$ are unital inclusions of $\ast$-algebras, respectively C$^\ast$-algebras, $\phi:\cA\lra\cD$ is a unital positive $\cB$-bimodule map and $a$ is a selfadjoint element of $\cA$, we will denote by $\cB\langle a \rangle$ the $\ast$-algebra generated in $\cA$ by $\cB$ and $a$ and by $\phi_a$, ``the $\cD$-distribution'' of $a$, that is the positive $\cB$-bimodule map $\phi_a:\bx\lra \cD$ defined by
 $\phi_a=\phi\circ \tau_a$
 where $\tau_a:\bx\lra \cA$ is the unique homomorphism such that $\tau_a(\X)=a$ and $\tau_a(b)=b$ for all $b\in\cB$.
\end{defn}

 We will denote by $\Sbd^0$ the set of elements from $\Sbd$ with the property that there exists some $M>0$ such that, for all  $b_1, \dots, b_n\in \cB$ we have that
 \begin{equation}\label{pr2} ||\mu(\X b_1 \X b_2\cdots \X b_n \X)||< M^{n+1}||b_1||\cdots ||b_n||.
 \end{equation}

 With the notations from Definition \ref{defn1}, an element $\mu\in\Sbd$ can be realized as $\phi_a$ for some C$^\ast$-algebra $\cA$ containing $\cB$ and $a$ a self-adjoint from $\cA$ if and only if $\mu\in\Sbd^0$ (see \cite{popa-vinnikov}, Proposition 1.2).

\subsection{Independences}
There are several independences important for noncommutative probability. We shall start with the oldest and best-known, Voiculescu's free independence. We present it in a C${}^*$-algebraic context.
\begin{defn}\label{defn31}
Let $\cB$ be a C$^\ast$-algebra, $\cB\subseteq\cA$ be a unital inclusion of $\ast$-algebras and $\phi\colon\mathcal A\to\mathcal B$ be a positive conditional expectation. A family $\{X_i\}_{i\in I}$ of selfadjoint elements from $\mathcal A$ is said to be {\em free} with respect to $\phi$ if
$$
\phi(A_1A_2\cdots A_n)=0
$$
whenever $A_j\in\mathcal B\langle X_{\epsilon(j)}\rangle$ with $\phi(A_j)=0$, where $\epsilon(j)\in I$, with $\epsilon(k)\neq\epsilon(k+1)$ for $1\leq k\leq n-1$.
\end{defn}

  Let now $N\in\mathbb{N}$ and $\{\mu_j\}_{j=1}^N$ be  a family of elements from $\Sigma_\cB$. We define their additive free  convolution the following way: 
Consider the symbols $\{X_j\}_{j=1}^N$; on the algebra $\cB\langle X_1, X_2, \dots, X_N \rangle$ take the conditional expectation $\mu$ such that $\mu\circ \tau_{X_j}=\mu_j$ and the mixed moments of $X_1, \dots, X_n$ are computed via the rules  from Definition \ref{defn31}. The free additive convolution of $\{\mu_j\}_{j=1}^N$ is the conditional expectation
  \[\boxplus_{j=1}^N \mu_j=\mu\circ\tau_{X_1+X_2+\cdots+X_N}:\cB\langle X_1+X_2+\dots+X_N\rangle\cong\bx\lra\cB.\]

 We have that $\boxplus_{j=1}^N \mu_j$ is also an element of $\Sb$: in \cite{speicherhab} 
it is shown that $\mu$, defined as above, is a positive conditional expectation, therefore so is $\mu\circ\tau_{X_1+X_2+\cdots+X_N}$.

Secondly, we give the op-valued equivalent of Speicher and Woroudi's Boolean independence \cite{SW} as it appears in \cite{popa-vinnikov}:
\begin{defn}\label{def-Boolean}
Let $\cB\subseteq\cA$, $\cB\subseteq \cD$ be unital inclusions of $\ast$-algebras, and $\phi\colon\mathcal A\to\mathcal D$ a unital completely positive $\mathcal B$-bimodule map. A family $\{X_i\}_{i\in I}$ of selfadjoint elements from $\mathcal A$ is said to be {\em Boolean independent} with respect to $\phi$ if
$$
\phi(A_1A_2\cdots A_n)=\phi(A_1)\phi(A_2)\cdots \phi(A_n)
$$
whenever $A_j$ is in the \emph{nonunital} $\ast$-subalgebra over  $\cB$ generated by $X_j$, where $\epsilon(j)\in I$ with $\epsilon(k)\neq\epsilon(k+1)$ for $1\leq k\leq n-1$.
\end{defn}

The definition of Boolean convolutions of distributions from $\Sigma_{\mathcal B:\mathcal D}$ is done similarly to free convolutions of distributions from $\Sigma_{\mathcal B}$ , as shown in \cite{popa-vinnikov}, by simply replacing free with Boolean independence. The reader will observe that this definition makes sense for $\mathcal B=\mathcal D$; the broader context that we provide adds in fact more depth to the theory. This will become clearer in the following definition, which essentially unites free and Boolean independence.

 We aim to extend the results that can be obtained for free independence to the case $\mu\in\Sbd$. In this setting, if $\cB$ is simply replaced by $\cD$, the resulting relation does not uniquely determine the joint moments of $X_1, \dots, X_n$. As shown in \cite{boca90,dykemablanchard}, a more suitable approach is the c-freeness (see also \cite{mvcomstoc} and \cite{bsk}).

 \begin{defn}\label{def-cfree}
 Let $\cB\subseteq\cA$, $\cB\subseteq \cD$ be unital inclusions of $\ast$-algebras, $E_\cB:\cA\lra\cB$ be a positive conditional expectation and $\theta:\cA\lra\cD$ be a unital $\cB$-bimodule map.

 The family $\{X_i\}_{i\in I}$ of selfadjoint elements from $\cA$ is said to be c-free with respect to $(\theta, \varphi)$ if
 \begin{enumerate}
 \item[(i)] the family $\{X_i\}_{i\in I}$ is free with respect to $E_\cB$
 \item[(ii)] $\theta(A_1A_2\cdots A_n)=\theta(A_1)\theta(A_2)\cdots \theta(A_n)$ for all $A_i\in\cB\langle X_{\epsilon(i)}\rangle$ such that $E_\cB(A_i)=0$ and $\epsilon(k)\neq\epsilon(k+1)$.
 \end{enumerate}
 \end{defn}
The reason for switching to the notation $E_\mathcal B$  will be seen later. We will consider the above relations in the framework of $\cB\subset\cD$ a unital inclusion of C$^\ast$-algebras and $\theta$ a completely positive map.As shown in \cite{boca90}, this setting (that includes $\Sbd$ is closed with respect to c-free convolution.

Next, we describe one of our main tools in analyzing distributions of sums of independent (freely, Boolean or c-freely) operator-valued random variables, namely {\em noncommutative} {\em sets and functions \cite{ncfound}.} We will use the terminology from \cite{ncfound} (see also \cite{taylor}, \cite{popa-vinnikov}, but translating the results in the different terminology from \cite{V1} is trivial.)

For a vector space $\cV$ over $\mathbb{C}$, we let
$M_{n\times m}(\cV)$, denote $n \times m$ matrices over $\cV$ and write $M_n(\cV)$ for $M_{n\times n}(\cV)$.
 We define the {\em noncommutative space} over $\cV$ by
$\ncspace{\cV} = \displaystyle\coprod_{n=1}^\infty M_n(\cV)$.
We call $\Omega \subseteq \ncspace{\cV}$ a {\em noncommutative set} if it is closed under direct sums.
Explicitly, denoting $\Omega_n = \Omega \cap M_n(\cV)$,
we have
\[
a \oplus b =\begin{bmatrix}
a & 0\\
0 & b\end{bmatrix}\in \Omega_{n+m}
\]
for all $a \in \Omega_n$, $b \in \Omega_m$.
Notice that matrices over $\mathbb{C}$ act from the right and from the
left on matrices over $\cV$  by the standard rules of
matrix multiplication.

Let $\cV$ and $\cW$ be vector spaces over $\mathbb{C}$, and let
$\Omega\subseteq\ncspace{\cV}$ be a noncommutative set. A
mapping $f \colon \Omega \to \ncspace{\cW}$ with
$f(\Omega_n) \subseteq M_n(\cW)$ is called a
\emph{noncommutative function} if $f$ satisfies the following two
conditions:
\begin{itemize}
\item $f$ \emph{respects direct sums}:
$f(a \oplus b) = f(a) \oplus f(b)$
for all $a,b \in\Omega$.
\item $f$ \emph{respects similarities}:
if $a \in \Omega_n$ and $s \in M_n(\mathbb{C})$ is invertible with
$sas^{-1} \in \Omega_n$, then
$f(sas^{-1}) = s f(a) s^{-1}$.
\end{itemize}
We will denote $f^{(n)} = f|_{\Omega_n} \colon \Omega_n \to M_n(\cW)$.
For convenience, we will refer sometimes, when there is no risk of confusion, to $f^{(n)}$ and $\Omega_n$ as the $n^{\rm th}$ {\em coordinate} of the respective noncommutative function and set.

A noncommutative set $\Omega\subseteq\ncspace{\cV}$ is called {\em upper admissible} if for all
$a \in \Omega_n$, $b \in \Omega_m$ and all $c \in M_{n\times m}(\cV)$,
there exists $\lambda \in \mathbb{C}$, $\lambda \neq 0$, such that
\[
\begin{bmatrix} a & \lambda c \\ 0 & b \end{bmatrix} \in \Omega_{n+m}.
\]
Non-commutative functions on upper-admissible sets admit a consistent differential calculus, including Taylor expansions by defining the (right) noncommutative difference-differential operators
by evaluating a noncommutative function on block upper triangular matrices (see again \cite{ncfound}). In this paper we will use only the following three upper admissible noncommutative sets ($\cA$ will denote a C$^\ast$-algebra):
\begin{enumerate}
\item
The set $\Nilp(\cA) = \coprod_{n=1}^\infty \Nilp(\cA;n)$; here the set $\Nilp(\cA;n)$ consists of
all $a \in M_n(A)$ such that $a^r = 0$ for some $r$, where we view $a$ as \emph{a matrix over
the tensor algebra} $\ten{\cA}$ of $\cA$ over $\mathbb{C}$
\item
Noncommutative balls $\ball(\cA,\rho) = \left\{a \in \ncspace{\cA} \colon \|a\| < \rho\right\}$
of radius $\rho>0$
over $\cA$ ($\cA$ could have been replaced by any operator space with the corresponding
operator space norm).
\item Noncommutative half-planes $\mathbb H^+(\ncspace{\cA}) = \left\{a \in \ncspace{\cA} \colon \Im a> 0\right\}$
over $\cA$. Here $\Im a=(a-a^*)/2i$ denotes the imaginary part of $a$. (We say that an element $b$ in a unital C${}^*$-algebra satisfies $b>0$ if there exists $\varepsilon\in(0,+\infty)$ so that $b\ge\varepsilon 1$, where here 1 is the unit of $\mathcal A$; of course, for an element $a\in\ncspace{\cA}$ to have imaginary part strictly greater than zero means simply that each ``coordinate'' has imaginary part strictly positive.) It has been first noted by Voiculescu \cite{V1} that $\mathbb H^+(\ncspace{\cA})$ is indeed a noncommutative set.
\end{enumerate}

\subsection{Op-valued distributions and properties of their transforms}

As for scalar-valued (non)commutative probability, there are transforms
that linearize different kinds of convolutions of operator-valued distributions. It turns out that these transforms can be described in terms of noncommutative functions defined on noncommutative spaces, associated to operator-valued distributions. First, we introduce some terminology and notations. If $\mathcal A\supseteq\mathcal B$ is a unital inclusion of $\ast$-algebras and $E_\mathcal B\colon\mathcal A\to\mathcal B$ is a conditional expectation, then $E_{M_n(\mathcal B)}=E_\mathcal B\otimes 1_n\colon M_n(\mathcal A)\to M_n(\mathcal B)$ is still a conditional expectation for any $n\in\mathbb N$ and any linear functional (in particular any trace $\tau$) on $\cB$ extends to $\tau\otimes{\rm tr}_n\colon\mathcal B\otimes M_n(\mathbb C)\to\mathbb C$, where ${\rm tr}_n$ is the canonical normalized trace on $M_n(\mathbb C)$. Note also that if $X, Y\in\mathcal A$ are free, boolean independent, respectively c-free with respect to $E_\cB$ and $\phi:\mathcal A\to \cD$ for some algebra $D$ containing $\cB$, then so are $X\otimes 1_n$ and $Y\otimes 1_n$ with respect to $E_{M_n(\mathcal B)}$ and $\phi\otimes 1_n$.

We recall that $a\in\mathcal A$ is called selfadjoint if $a=a^\ast$. Any element $a\in\mathcal A$ in a $*$-algebra can be written uniquely $a=\Re a+i\Im a$, where $\Re a=(a+a^\ast)/2,$ $\Im a=(a-a^\ast)/(2i)$ are selfadjoint.
For any given $*$-algebra $\mathcal A$ on which a notion of
positivity coherent with the star operation has been defined, we
shall denote the upper half-plane of $\mathcal A$ by
$$
\mathbb H^+(\mathcal A)=\{a\in\mathcal A\colon\Im a>0\}
,$$
and $\mathbb H^-(\mathcal A)=-\mathbb H^+(\mathcal A)$. (Thus,
$\mathbb H^+(\ncspace{\cA}) = \coprod_{n=1}^\infty\mathbb H^+(M_n(\mathcal A));$ see also \cite{V1}.)
We note that $a\in\mathbb H^+(\mathcal A)\implies a^\ast\in\mathbb H^-(\mathcal A)$.

A useful generalization, noted by Voiculescu, of the fact that the operation of taking inverse changes the imaginary part of a complex number from positive to negative and vice-versa is the following implication, which holds in any unital $*$-algebra in which analytic functional calculus is available:
\begin{equation}\label{inverse}
a\in\mathbb H^+(\mathcal A)\implies a^{-1}\in\mathbb H^-(\mathcal A).
\end{equation}
(Note that the invertibility of $a$ is part of the statement.) Indeed, by writing $u=\Re a,v=\Im a$, and $a=u+iv$, we have
$$
a=u+iv=u+i(\sqrt{v})^2=\sqrt{v}[(\sqrt{v})^{-1}u(\sqrt{v})^{-1}+i]\sqrt{v}.
$$
The ability to take square root is guaranteed by the analytic functional calculus and the fact that $v>\varepsilon 1_\mathcal A$ for some $\varepsilon>0$. As $(\sqrt{v})^{-1}u(\sqrt{v})^{-1}$ is selfadjoint, it is clear that $i$ does not belong to its spectrum, so
$(\sqrt{v})^{-1}u(\sqrt{v})^{-1}+i$ is invertible in $\mathcal A$.
Invertibility of $\sqrt{v}$, again guaranteed by its strict positivity and the existence of analytic functional calculus, implies that $a$ is itself invertible in $\mathcal A$. Now writing its inverse gives
\begin{eqnarray*}
a^{-1} & = & \left(\sqrt{v}[(\sqrt{v})^{-1}u(\sqrt{v})^{-1}+i]\sqrt{v}\right)^{-1}\\
& = & (\sqrt{v})^{-1}[(\sqrt{v})^{-1}u(\sqrt{v})^{-1}+i]^{-1}(\sqrt{v})^{-1}\\
& = & (\sqrt{v})^{-1}\left[[(\sqrt{v})^{-1}u(\sqrt{v})^{-1}]^2+1\right]^{-1}[(\sqrt{v})^{-1}u(\sqrt{v})^{-1}-i](\sqrt{v})^{-1}\\
& = & \underbrace{(\sqrt{v})^{-1}\left[[(\sqrt{v})^{-1}u(\sqrt{v})^{-1}]^2+1\right]^{-1}[(\sqrt{v})^{-1}u(\sqrt{v})^{-1}](\sqrt{v})^{-1}}_{=c}\\
& & +i\underbrace{\left\{-(\sqrt{v})^{-1}\left[[(\sqrt{v})^{-1}u(\sqrt{v})^{-1}]^2+1\right]^{-1}(\sqrt{v})^{-1}\right\}}_{=d}.
\end{eqnarray*}
Analytic functional calculus guarantees that $c=c^*,d=d^*$ and $-d>0$. The uniqueness of the expansion into real and imaginary part guarantees that $c=\Re(a^{-1})$, $d=\Im(a^{-1})$, so our claim is proved.

We indicate next how an operator-valued distribution can be encoded by noncommutative functions.
Note that if $\mu\in\Sigma_{\cB:\mathcal D}$, then  $(\mu\otimes 1_n)\in\Sigma_{M_n(\cB):M_n(\mathcal D)}$. Moreover, (see \cite{ncfound}, \cite{V1}) for

    \[
 b=\left[
 \begin{array}{ccccc}
 0&b_1&0&\dots&0\\
 0&0&b_2&\dots&0\\
 \dots&\dots&\dots&\dots&\dots\\
 0&0&0&\dots&b_n\\
 0&0&0&\dots&0
\end{array}
 \right]
 \in M_{n+1}(\cB)
 \]
 we have that
\begin{equation}\label{momente}
  ( \mu\otimes 1_{n+1})\bigl([\X b]^n)=\left[
\begin{array}{cccc}
0&\dots&0&\mu(\X b_1 \X b_2\dots \X b_n)\\
0&\dots&0&0\\
\dots&\dots&\dots&\dots\\
0&\dots&0&0
\end{array}
\right],
\end{equation}
so $\mu$ is completely characterized by the sequence $\bigl\{({\mu}\otimes\id_{n}) \bigl([\X\cdot b]^m\bigl)\bigr\}_{m, n}$.

Using this observation, we shall indicate how all the information describing a distribution can be encapsulated in a fully matricial (or noncommutative) function. For a given $\mu\in\Sbd$, we define its {\em moment-generating series} as the non-commutative function of components
\begin{eqnarray*}
M^{(n)}_\mu(b)=\sum_{k=0}^\infty(\mu\otimes 1_n)([\X \cdot b]^k)
&=&
1_n+(\mu\otimes 1_n)(\X\cdot b)\\
&&\hspace{-1.5cm}
+(\mu\otimes 1_n)(\X\cdot  b\cdot\X\cdot  b)+\cdots.
\end{eqnarray*}

 As shown in \cite{popa-vinnikov}, $M_\mu$ is always well-defined on $\Nilp(\cB)$. Moreover, from Engel's Theorem, $a\in\Nilp(M_n(\cB))$ if and only if $TaT^{-1}$ is strictly upper
triangular for some $T\in GL(n)$, therefore $T[M_\mu(b)-\mathds{1}]T^{-1}$ is also upper-triangular hence $M_\mu(b)-\mathds{1}\in\Nilp(\cD)$. As in \cite{popa-vinnikov}, we used the symbol $\mathds{1}$ for $\id_n$ on each component from $M_n(\cB)$. If $\mu\in\Sbd^0$, then $M_\mu$ is also well-defined on a small non-commutative ball from $\ncspace{\cB}$ which is mapped in a non-commutative ball from $\ncspace{\cD}$.

 For $\nu\in\Sb$ and $\mu\in\Sbd$ we define their $R$-, respectively {\em $B$-transforms} via the functional equations
 \begin{eqnarray}
 M_\nu(b)-\mathds{1}&=&R_\nu\left(b\cdot M_\nu(b)\right)\label{rtransform}\\
 M_\mu(b)-\mathds{1}&=&B_\mu(b)\cdot M_\mu(b)\label{btransform}
 \end{eqnarray}
 In \cite{popa-vinnikov} is is shown that the $R$- and $B$-transforms are non-commutative functions well-defined on $\Nilp(\cB)$. If $\nu\in\Sb^0$, respectively $\mu\in\Sbd^0$, then $R_\nu$ and $B_\mu$ are also well-defined in some non-commutative balls from $\ncspace{\cB}$.

The main reason for which we have introduced the $R$ and $B$-transforms is their linearizing property: we have
\begin{equation}\label{R-transform}
R_{\mu\boxplus\nu}(b)=R_\mu(b)+R_\nu(b),\quad \mu,\nu\in\Sigma_\mathcal B,
\end{equation}
\begin{equation}\label{B-transform}
B_{\mu\uplus\nu}(b)=B_\mu(b)+B_\nu(b),\quad \mu,\nu\in\Sigma_{\mathcal B:\mathcal D}.
\end{equation}
The first result is due to Voiculescu \cite{V*}, and the second to Popa \cite{mv-bool}.

We warn the reader that the other version of the $R$-transform,
defined below, namely the original one of Voiculescu, as well as
the one used by Dykema (that we will call here $\mathcal R$), is related to this version by a simple
multiplication to the right with the variable $b$:
$R_\mu(b)=\mathcal R_\mu(b)b$.

Depending on the functional context, sometimes it is convenient to use slight variations of these transforms, which benefit of similar properties. We start with the most straightforward: following \cite{P}, we introduce the {\em shifted moment generating function} $\mathfrak H_\mu$ of $\mu\in\Sbd$ as the non-commutative function of components
\[
\mathfrak H^{(n)}_\mu(b) = b\cdot M^{(n)}_\mu(b).
\]
When there is no risk of confusion, we will denote $\mathfrak H^{(1)}_\mu$ also by $\mathfrak H_\mu$.

For a better understanding of the way noncommutative functions generalize the classical functions associated to probability distributions, it will be convenient to express the transform $\mathfrak H_\mu$ in terms of the {\em generalized resolvent} or {\em operator-valued Cauchy transform}. Several properties of scalar-valued Cauchy transforms (easily proved, but available also in \cite[Chapter III]{akhieser}) are preserved when we pass to the operator-valued context. We shall express the operator-valued Cauchy transform first in terms of random variables. Suppose that $\cB\subset\mathcal A$, $\cB\subset\cD$ are inclusions of unital C$^\ast$-algebras and $\phi:\cA\to\cD$ is a completely positive $\cB$-bimodule map.  For any fixed $X=X^*\in \mathcal A$ we let $\mathcal G_X=(G_X^{(n)})_n$,
\begin{align*}
G_X^{(n)}&\colon\mathbb H^+(M_n(\mathcal B))\to\mathbb H^-(M_n(\mathcal D)),\\
G_X^{(n)}(b)&=\phi_n[(b-X\otimes \id_n)^{-1}],
\end{align*}
where $\phi_n$ is the extension of $\phi$ from $M_n(\mathcal A)$ to $M_n(\mathcal D)$. For $n=1$ we shall denote $G_X^{(1)}(b)$ simply by $G_X(b)$. Let us first remark that this expression makes sense: indeed, since $\Im b>0$ and $X=X^\ast$, it follows that $\Im(b-X\otimes \id_n)>0$, so, as noted in \eqref{inverse}, $b-X\otimes \id_n$ is invertible in $\mathcal A$ and its inverse has  strictly negative imaginary part. Since $\phi$ is completely positive, it follows in addition that $G_X^{(n)}$ takes values in the lower matricial half-plane.

 We observe that, whenever $\|b^{-1}\|<1/\|X\|$, we can write $G_X(b)=\sum_{n=0}^\infty b^{-1}\phi[(Xb^{-1})^n]$ as a convergent series. Thus, it follows easily that for $\mu\in\Sbd^0$ we shall write
 \[
 G^{(n)}_\mu(b)=\sum_{n=0}^\infty (\mu\otimes\id_n)(b^{-1}(\mathcal X\cdot b^{-1})^n)
 =(\mu\otimes\id_n)[(b-\mathcal X)^{-1}];
  \]
  (of course, these equalities require that we consider an extension of $\mu$ to $\mathcal B\langle\langle\mathcal X\rangle\rangle$, the algebra of formal power series generated freely by $\mathcal B$ and the selfadjoint symbol $\mathcal X$). This also indicates a very important equality, namely, for $\mu\in\Sbd^0$
\begin{equation}\label{G-frakH}
{\mathcal G}_\mu(b^{-1})=\mathfrak H_\mu(b),\quad b\in\mathbb H^+(M_n(\cB)).
\end{equation}
Moreover, ${\mathcal G}_\mu(b^*)=[{\mathcal G}_\mu(b)]^*$ extends ${\mathcal G}_\mu$ to the lower half-planes, analytically through points $b$ with inverse of small norm.

It has been shown by Voiculescu \cite{V1} that
$\mathbb H(\ncspace{\cB})=\coprod_{n=1}^\infty\mathbb H^+(M_n(\cB))$ and ${\mathcal G}_X$
with the structures defined above are fully matricial sets
and functions. It is easy to observe that the same is true for $\mathcal F_X$, the reciprocals of ${\mathcal G}_X$, namely
\[ F_X(b)=[G_X(b)]^{-1}, \hspace{.6cm}
F_X^{(n)}(b)=[G_X^{(n)}(b)]^{-1}.\]
\begin{remark}\label{about-F}
When $X$ is a $\mathcal B$-valued selfadjoint random variable, the transform $\mathcal F$ has many properties in common with its scalar-valued analogue. First of all, it follows straightforwardly from the similar property of $\mathcal G$ that $F^{(n)}_X$ necessarily maps $\mathbb H^+(M_n(\cB))$ into itself. Moreover, under the condition that $\cB$ has a rich enough collection of positive linear functionals (for example if it is a $C^\ast$-algebra) we always have $\Im F^{(n)}_X(b)\ge\Im b$ for all $b\in\mathbb H^+(M_n(\cB))$. Indeed, for $n=1$, for any fixed positive linear functional $f$ on $\mathcal B$, let us define $F_{X,f}(z)=f(F_X(\Re b+z\Im b))$. Clearly, since $\Im b>0$, we have that $F_{X,f}\colon\mathbb H^+(\mathbb C)\to\mathbb H^+(\mathbb C)$. In addition,
\begin{eqnarray*}
\lim_{y\to+\infty}\frac{F_{X,f}(iy)}{iy} & = &
\lim_{y\to+\infty}f\left(\left[\sum_{n=0}^\infty(\Im b)^{-1}
\phi\left[\frac{1}{(iy)^n}((X-\Re b)(\Im b)^{-1})^n\right]
\right]^{-1}\right)\\
& = & f\left(\Im b\right)>0.
\end{eqnarray*}
Thus \cite[Chapter III]{akhieser} there exists a positive compactly supported Borel measure $\rho$ on the real line of mass $1/f\left((\Im b)\right)$ so that $F_{X,f}(z)=\frac{1}{G_\rho(z)}=\left[\int_\mathbb R\frac{1}{z-t}\,d\rho(t)\right]^{-1}$ for all $z$ in the upper half-plane. The Nevanlinna representation of $F_{X,f}$ implies that
$\Im F_{X,f}(z)\ge f\left(\Im b\right)\Im z$ for all  $z\in\mathbb H^+(\mathbb C)$, and so, since this holds for all positive linear functionals $f$ on $\mathcal B$, $\Im F_X(b)\ge\Im b$ for all $b\in\mathbb H^+(\mathcal B)$. Moreover, we note that if equality holds for a given $b_0$, then $X-\Re b_0$ must be a multiple of the identity of $\mathcal B.$ The argument for a general component $F^{(n)}_X$ is analogous.
\end{remark}

We would like also to mention the connection between $\mathcal F$ and $B$:
\begin{equation}\label{F-B}
\mathds{1}-\mathcal F_\mu(b^{-1})b=B_\mu(b),\quad b^{-1}\in\mathbb H^+(\ncspace{\cB}).
\end{equation}

For classical measures it is known that weak convergence to a finite measure is equivalent to uniform convergence on compact sets for the Cauchy transforms to the Cauchy transform of the limit, and, if all measures involved are compactly supported, these two statements are equivalent to the convergence of moments (we say $\sigma_n\to\sigma$ in moments if $\int t^jd\sigma_n(t)$ converges to $\int t^jd\sigma(t)$ for any $j\in\mathbb N$.) We shall provide below two versions of this result for operator-valued distributions.

First, let us define convergence in moments for an operator-valued distribution.
\begin{defn}
Given a sequence of distributions $\mu_n\in\Sigma_{\mathcal B:\mathcal D}$, we say that
\begin{itemize}
\item[(a)] $\mu_n$ converges to $\mu\in\Sigma_{\mathcal B:\mathcal D}$
{\em pointwise} in moments if for any $\varphi\in\mathcal D^*$ we have
\[
\lim_{n\to\infty}\mu_n(\mathcal Xb_1\mathcal Xb_2\cdots b_j\mathcal X)=\mu(\mathcal Xb_1\mathcal Xb_2\cdots b_j\mathcal X),
\]
for all $b_1,\dots,b_j\in\mathcal B$;
\item[(b)]$\mu_n$ {\em norm}-converges to $\mu\in\Sigma_{\mathcal B:\mathcal D}$ in moments if
$$
\lim_{n\to\infty}\sup_{\|b_k\|=1,1\le k\le j}\|\mu_n(\mathcal Xb_1\mathcal Xb_2\cdots\mathcal Xb_j\mathcal X)-\mu(\mathcal Xb_1\mathcal Xb_2\cdots\mathcal Xb_j\mathcal X)\|=0,
$$
for all $j\in\mathbb N$.
\end{itemize}
\end{defn}
Remark first that for finite dimensional algebra $\mathcal B$ the two notions are equivalent. Also, remark that condition [(b)] is equivalent to
\begin{itemize}
\item[(b$^\prime$)] $\displaystyle \lim_{n\to\infty}M_{\mu_n}(b)=M_{\mu}(b),$ \ for all $b\in\Nilp(\cB)$.
\end{itemize}
In this paper we will mainly be interested in norm-convergence of moments.

We note next the following simple remark:
\begin{remark}\label{pointwise}
Assume that $\mathcal B\subseteq\mathcal D$ is a unital inclusion of von Neumann algebras and $\{\mu_n\}_{n\in\mathbb N}$ is sequence from $\Sbd^0$  Then $\mu_n$ converges pointwise in moments to $\mu\in\Sbd^0$  if and only if $\mathcal G_{\mu_n}$ converges pointwise to $\mathcal G_\mu$ on $\mathbb H^+(\ncspace{\cB})$.
\end{remark}

\begin{proof}
First, let us assume that $\mu_n$ converges to $\mu$ pointwise in moments. It suffices to prove that for any state $\varphi$ on $M_m(\cD)$ we have that
\[
\lim_{n\lra\infty}\varphi\bigl(G_{\mu_n}^{(m)}(b)\bigr)=\varphi\bigl(G_{\mu}^{(m)}(b)\bigr).
\]
Let $\varphi$ as above and $b\in \mathbb{H}^{+}(M_m(\cB))$. The function $\mathbb{H}^+(\mathbb{C})\ni z\mapsto\varphi\bigl(G^{(m)}_{\mu_n}(\Re b +z\Im b)\bigr)$ is the Cauchy transform of some positive real measure $\sigma_{n,\varphi,b}$. Indeed,
\[
\varphi(G^{(m)}_{\mu_n}(\Re b+z\Im b))=\varphi((\mu_n\otimes1_m)\left[(\Re b+z\Im b-\mathcal X)^{-1}\right]),
\]
 and since $\varphi,\mu_n\otimes1_m,\Im b$ and $\Im z$ are all positive, it follows that $\Im\varphi(G^{(m)}_{\mu_n}(\Re b+z\Im b))<0$. Moreover,
\begin{eqnarray}
\lim_{z\to\infty}z(\mu_n\otimes1_m)\left[(\Re b+z\Im b-\X)^{-1}\right]&=&\nonumber\\
&&\hspace{-3cm}\lim_{z\to\infty}\sum_{j=0}^\infty(\Im b)^{-1}(\mu_n\otimes1_m)\left[\frac{((\Re b-\X)(\Im b)^{-1})^j}{z^j}\right]\label{eq5}\\
&&\hspace{-3cm}=(\Im b)^{-1}\nonumber
\end{eqnarray}
(the limit is in the norm topology of $M_m(\cD)$; these expressions make sense for $|z|$ large enough because of the exponential growth condition) Hence applying $\varphi$ and using that $\Im b>0$ gives the result.
 Thus, according to \cite{akhieser}, Chapter III, we have that
  \[
  \varphi(G^{(m)}_{\mu_n}(\Re b+z\Im b))=\int_\mathbb R(z-t)^{-1}\,d\sigma_{n,\varphi,b}(t),
  \]
   with $\sigma_{n,\varphi,b}(\mathbb R)=\varphi((\Im b)^{-1})$. We observe in addition that
   $$\int_\mathbb Rt^j\,d\sigma_{n,\varphi,b}(t)=\varphi\left((\mu_n\otimes1_m)\left[{(\Im b)^{-1}((\Re b-\mathcal X)(\Im b)^{-1})^j}\right]\right),$$
    so by our hypothesis and the continuity of the multiplication with the constant $(\Im b)^{-1}$ we obtain that the moments of $\sigma_{n,\varphi,b}(t)$ converge. Normality of the family $\{\varphi(G^{(m)}_{\mu_n}(\Re b+z\Im b))\}_n$ guarantees that a weak limit of this sequence of measures exists, and the limit has the prescribed moments. We conclude that $\varphi(G^{(m)}_{\mu_n}(\Re b+z\Im b))$ converges to $\varphi(G^{(m)}_{\mu}(\Re b+z\Im b))$ uniformly on compacts of $\mathbb H^+(\mathbb C)$ for any fixed $b\in\mathbb H^+(M_m(\cB))$.

 Assume next that $\lim_{n\to\infty}G^{(m)}_{\mu_n}(b))=G^{(m)}_\mu(b))$ for any $m$ and any $b\in\mathbb H^+(M_m(\cB))$. We define again $\sigma_{n,\varphi,b}$ as above and observe that $\varphi(G^{(m)}_{\mu_n}(\Re b+z\Im b))$ is the Cauchy transform of this measure and it converges uniformly on compacts of the complex upper half-plane to the Cauchy transform of a limit measure $\sigma_{\varphi,b}$. By applying $\varphi$ to equation \eqref{eq5} we obtain that $\sigma_{\varphi,b}$ has moments $\varphi\left((\mu\otimes1_m)\left[{(\Im b)^{-1}((\Re b-\mathcal X)(\Im b)^{-1})^j}\right]\right)$, so that
\begin{eqnarray*}
\lefteqn{\lim_{n\to\infty}\varphi\left((\Im b)^{-1}(\mu_n\otimes1_m)\left[(\Re b-\mathcal X)(\Im b)^{-1}(\Re b-\mathcal X)\cdots(\Re b-\mathcal X)\right](\Im b)^{-1}\right)=}\\
& & \varphi\left((\Im b)^{-1}(\mu\otimes1_m)\left[(\Re b-\mathcal X)(\Im b)^{-1}(\Re b-\mathcal X)\cdots(\Re b-\mathcal X)\right](\Im b)^{-1}\right),
\end{eqnarray*}
from which we obtain pointwise convergence for all symmetric moments of the fully matricial extensions of $\mu_n$ to $\mu$, hence for all moments of $\mu_n$ and $\mu$.
\end{proof}

\begin{defn}\label{marg-unif}
 A sequence $\{\mu_n\}_n$ from $\Sbd^0$ is said to be uniformly bounded (by $M$) if there exists some constant $M>0$ such that for all $n,p$ and all $b_1, \dots, b_p\in \cB$ we have that
 \[
 \| \mu_n(\X\cdot b_1\cdot \X\cdots b_p \cdot \X)\|<M^{p+1}\|b_1\|\cdots\|b_p\|.
 \]
\end{defn}
Note that, according to \cite{popa-vinnikov}, Proposition 1.2, the above condition holds true (with the same $M$) for $\mu_n\otimes1_m$ and $b_1,\dots,b_p\in M_m(\cB)$.

 For the next Proposition we shall find useful the following (purely Banach space) results, which can be found in the first chapters of \cite{isidro}(Theorems 1.5 and 1.6). Let $E$ and $E_1$ be complex Banach spaces, $D\subset E$ and $D_1\subset E_1$ be bounded domains. Following \cite{isidro}, we denote by $\mathrm{Hol}(D,D_1)$ the set of holomorphic mappings from $D$ into $D_1$, that is, functions $f$ for which $f(a+h)=\sum_{n=0}^\infty f^{(n}_a(h,\dots,h)$ on a neighborhood of $a$, for any $a$ in $D$; where
\[
f^{(n}_a(h_1,\dots,h_n)=\frac{1}{n!}\frac{\partial^n}{\partial t_1\cdots\partial t_n}f(a+t_1h_1+\dots+t_nh_n)
\]
is a continuous $n$-linear map from $E^n$ to $E_1$. We shall also denote $B\subset\subset D$ if $B$ is a subset of $D$ with the additional property that the norm distance from $B$ to $\partial D$ is strictly positive.

\begin{thm}\label{anal-conv}
Let $(f_j)_{j\in J}$ be a net in $\mathrm{Hol}(D,D_1)$ and $f\in\mathrm{Hol}(D,D_1)$, and $B\subset\subset D$ be a ball centered at $a\in D$. The following statements are equivalent:
\begin{enumerate}
\item The net $(f_j)_{j\in J}$ is uniformly convergent to $f$ on $B$;
\item For all $k\in\mathbb N$ we have $\lim_{j\in J}\|f^{(k}_{j,a}-f^{(k}_a\|=0.$
\end{enumerate}
\end{thm}
\begin{thm}\label{indiferent-conv}
Let $(f_j)_{j\in J}$ be a net in $\mathrm{Hol}(D,D_1)$.  For any two balls $B_1, B_2\subset\subset D$ the following statements are equivalent:
\begin{enumerate}
\item $f_j\to f$ relative to $\|\cdot\|_{B_1}$;
\item $f_j\to f$ relative to $\|\cdot\|_{B_2}$.
\end{enumerate}
where $\|f\|_B=\sup_{x\in B}\|f(x)\|$.
\end{thm}
We would like to emphasize that in Remark \ref{pointwise} we do NOT require that the sequence $\{\mu_n\}$ is uniformly bounded. However, in order to be able to prove the similar result for norm-convergence of moments in the most general C$^\ast$-algebraic context, we will have to require that.

\begin{prop}\label{continuity}
Assume $\mathcal B\subseteq\mathcal D$ is a unital inclusion of C$^\ast$-algebras and that the sequence $\{\mu_n\}_{n\in\mathbb N}\subset\Sbd^0$ is uniformly bounded. The following statements are equivalent:
\begin{enumerate}
\item $\{\mu_n\}_{n\in\mathbb N}$ norm-converges in moments to some $\mu\in\Sbd^0$;
\item for all positive integers $m$, $G^{(m)}_{\mu_n}$ converges uniformly to $G^{(m)}_\mu$ on balls in $\mathbb{H}^{+}(M_m(\cB))$ which lay at positive distance from $\partial\mathbb H^+(M_m(\cB))$.
\end{enumerate}
\end{prop}

\begin{proof} 
 Let us denote by $M$ a common bound for $\{\mu_n\}_n$ and $\mu$ as in Definition \ref{marg-unif}.
 
 For (2)$\Rightarrow$(1), note first that for all $m$, $\mathfrak{H}^{(m)}_\mu$ is well-defined in $B_{\frac{1}{2M}}(0)$, the ball of center zero and radius $\frac{1}{2M}$ from $M_m(\cB)$. Moreover, $\mathfrak{H}^{(m)}_\mu\in\mathrm{Hol}(B_{\frac{1}{2M}}(0), M_m(\cD))$ and
\begin{equation}\label{symm}
\mathfrak{H}_{\mu,0}^{(m)(k}(h_1,\dots,h_k)=\frac{1}{k!}\sum_{\sigma\in S_k}(\mu\otimes1_m)\left(
h_{\sigma(1)}\X h_{\sigma(2)}\cdots \X h_{\sigma(k)}
\right)
\end{equation}
where $S_k$ denotes the symmetric group with $k$ elements.

Fix now $b\in\mathbb{H}^{-}(M_m(\cB))$ with $\|b\|<\frac{1}{2M}$. Then there exist some small $R>0$ such that $B_R(b^{-1})\subset\subset\mathbb{H}^{+}(M_m(\cB))$. Since $\mathcal{G}_\mu(h^{-1})=\mathfrak{H}_\mu(h)$, it follows that there exists some $r>0$ such that $B_r(b)\subset\subset B_{\frac{1}{2M}}(0)$ and $\{\mathfrak{H}_{\mu_n}\}_n$ converges uniformly to $\mathfrak{H}_\mu$ on $B_r(b)$. Applying Theorem \ref{indiferent-conv}, we have that $\{\mathfrak{H}_{\mu_n}\}_n$ converges uniformly to $\mathfrak{H}_\mu$ on $B_{\frac{1}{4M}}(0)$, hence
\[
\|\mathfrak{H}^{(m),(k}_{\mu_n,0}-\mathfrak{H}^{(m),(k}_{\mu,0}\|\lra 0
\]
as $k$-linear maps from $M_m(cB)^k$ to $M_m(\cD)$. But (\ref{symm}) gives 
\[
\mathfrak{H}^{(m),(k}_{\mu,0}(h,\dots,h)=(\mu\otimes1_m)\left(h\X h\cdots \X h\right)
\]
and, since $m$ is arbitrary, equation \eqref{momente} allows us to conclude.

For (1)$\Rightarrow$(2), we will use the result from \cite{popa-vinnikov}, Proposition 1.2 namely that if $\mu\in\Sbd^0$ then there exist a C$^\ast$-algebra $\cA$ containing $\cB$, some selfadjoint $X\in\cA$ and a unital completely positive $\cB$-bimodule map $\phi:\cA\lra\cD$ such that for all noncommutative polynomials $f$ with coefficients in $\cB$ we have that $\mu(f(\X))=\phi(f(X))$.

Fix now $b_0\in\mathbb{H}^{+}(M_m(\cB))$. Since $\Im(X-b_0)=-\Im(b_0)<0$, we have that $X-b_0$ is invertible in $M_m(\cA)$ and
\begin{eqnarray}
\|(X-b_0)^{-1}\| & = & \|(i\Im b_0+\Re b_0-X)^{-1}\|\label{maj1}\\
&&\hspace{-1.5cm}= \|(\Im b_0)^{-1/2}(i+(\Im b_0)^{-1/2}(\Re b_0-X)(\Im b_0)^{-1/2})^{-1}(\Im b_0)^{-1/2}\|\nonumber\\
& \leq & \|(\Im b_0)^{-1}\|\|(i+(\Im b_0)^{-1/2}(\Re b_0-X)(\Im b_0)^{-1/2})^{-1}\|\nonumber\\
& \leq & \|(\Im b_0)^{-1}\|.\nonumber
\end{eqnarray}
(We have used here the fact that $M_m(\cB)$ is a C${}^*$-algebra, the fact that $\Im b_0$ is selfadjoint, as well as the fact that $i+(\Im b_0)^{-1/2}(\Re b_0-X)(\Im b_0)^{-1/2}$ is normal, so that one can apply continuous functional calculus to it.) Note that the above majorization is independent of $X$.

Also, for $\|h\|<\frac{1}{\|\Im{b_0}^{-1}\|}$, we have 
\begin{eqnarray*}
(b_0+h-X)^{-1}&=&[h+(b_0-X)]^{-1}\\
&=&(X-b_0)^{-1}[h(X-b_0)^{-1}-1_m]^{-1}\\
&=&(b_0-X)^{-1}
\sum_{n=0}^\infty\left[h(X-b_0)^{-1}\right]^n.
\end{eqnarray*}

Since $\phi$ is unital and completely positive, we have that $\|\phi\|_{\rm cb}=\|\phi(1)\|=1$, hence we can apply $\phi$ to the power series development of $(b_0+h-X)^{-1}$. It follows that
\begin{eqnarray*}
G^{(m)}_\mu(b_0+h)&=&\phi_m\left((b_0+h-X)^{-1}\right)\\
&=&\sum_{n=0}^\infty \phi_m\left((b_0-X)^{-1}[h(X-b_0)^{-1}]^n\right),
\end{eqnarray*}
hence
\begin{eqnarray*}
\|G_\mu(b)\| 
& = & \|\phi[(b-X)^{-1}]\|\\
& \leq & \sum_{n=0}^\infty\left\|\phi\left[(X-b_0)^{-1}\left[(b-b_0)(X-b_0)^{-1}\right]^n\right]\right\|\\
& \leq &\sum_{n=0}^\infty\|\phi\|_{\rm cb}\|(X-b_0)^{-1}\|^{n+1}\|b-b_0\|^n\\
& \le & \frac{\|(\Im b_0)^{-1}\|}{1-\|(\Im b_0)^{-1}\|\|b-b_0\|}.
\end{eqnarray*}
Also,
\[
G_{\mu,b_0}^{(m),(k}=\frac{1}{k!}\sum_{\sigma\in S_k}
\phi_m\left(
(b_0-X)^{-1}h_{\sigma(1)}(X-b_0)^{-1}h_{\sigma(2)}\cdots h_{\sigma(k)}(X-b_0)^{-1}
\right)
\]
 
\noindent By the above, it is easy to observe that each of these $k$-linear functionals is bounded in norm by $(k+1)!(\|(\Im b)^{-1}\|^k+\|(\Im b)^{-1}\|^k\|\Re b\|)$.

To simplify the notations, we will prove (2) for $m=1$; for an arbitrary $m$ all the computations are similar, using the matricial extensions of $\mu$ and $\{\mu_n\}_n$.

 We shall prove that there exists a point, namely $Q=(1+2M)i$, around which there exists a ball of radius $1/2$ on which $G_{\mu_n}$ converges to $G_\mu$ uniformly in norm. Then we shall use Theorem \ref{indiferent-conv} to argue that this implies uniform convergence on any ball $B\subset\subset\mathbb H^+(\cB)$ - the result will be proved by using Theorem \ref{anal-conv}. Indeed, let us start by observing that Theorems \ref{indiferent-conv} and \ref{anal-conv} indeed apply to our functions whenever we restrict them to $\mathbb H^+(\cB)+ic$ for any $c>0$. For a fixed pozitive integer $k$ we have that
\begin{eqnarray*}
\lefteqn{\mu\left[(Q-\mathcal X)^{-1}h_1(Q-\mathcal X)^{-1}\cdots h_k(Q-\mathcal X)^{-1}\right]=}\\
& & \frac1{Q^{k+1}}\mu
\left[\left(1-\frac{\mathcal X}{Q}\right)^{-1}h_1\left(1-\frac{\mathcal X}{Q}\right)^{-1}\cdots h_k\left(1-\frac{\mathcal X}{Q}\right)^{-1}\right]\\
& = & 
\frac1{Q^{k+1}}
\mu
\left[\sum_{m_1,\dots,m_{k+1}=0}^\infty\frac{\mathcal X^{m_1}}{Q^{m_1}}h_1
\frac{\mathcal X^{m_2}}{Q^{m_2}}\cdots h_k\frac{\mathcal X^{m_{k+1}}}{Q^{m_{k+1}}}\right]\\
& = & 
\frac1{Q^{k+1}}\mu
\left[\sum_{m_1,\dots,m_{k+1}=0}^\infty\frac{1}{Q^{m_1+m_2+\cdots+m_{k+1}}}
{\mathcal X^{m_1}}h_1{\mathcal X^{m_2}}\cdots h_k{\mathcal X^{m_{k+1}}}\right]\\
& = &
 \mu\left[\sum_{q=0}^\infty\frac1{Q^{k+1+q}}\sum_{m_1+\dots+m_{k+1}=q}{\mathcal X^{m_1}}h_1{\mathcal X^{m_2}}\cdots h_k{\mathcal X^{m_{k+1}}}\right]\\
& = & 
\sum_{q=0}^\infty\frac1{Q^{k+1+q}}\sum_{m_1+\dots+m_{k+1}=q}\mu
\left[{\mathcal X^{m_1}}h_1{\mathcal X^{m_2}}\cdots h_k{\mathcal X^{m_{k+1}}}\right].
\end{eqnarray*}
We note that the majorization
\begin{eqnarray*}
\|\mu\left[{\mathcal X^{m_1}}h_1{\mathcal X^{m_2}}\cdots h_k{\mathcal X^{m_{k+1}}}\right]\| & \leq & M^{m_1+m_2+\cdots+m_{k+1}}\|h_1\|\cdots\|h_k\|\\
& = & M^q\|h_1\|\cdots\|h_k\|
\end{eqnarray*}
guarantees that the last sum above is majorized in norm by $M^q(q+k)^k$; since $|Q|=1+2M,$ the convergence of this expression is not a problem. Note that all the above estimates hold true also for $\{\mu_n\}_n$. We claim that
$$
\lim_{n\to\infty}\left\|\mu_n\left[(Q-\mathcal X)^{-1}h_1(Q-\mathcal X)^{-1}\cdots h_k(Q-\mathcal X)^{-1}\right]\right.
$$
$$\left.-\mu\left[(Q-\mathcal X)^{-1}h_1(Q-\mathcal X)^{-1}\cdots h_k(Q-\mathcal X)^{-1}\right]\right\|=0.
$$
Indeed, let $\varepsilon>0$ be fixed. By the choice of $Q$ (and the norm convergence of the last series from above)
it follows that there exists a positive integer $q(\varepsilon,Q)$ {\em not depending on }$n$, so that
$$
\left\|\sum_{q=r}^\infty\frac1{|Q|^{k+1+q}}\sum_{m_1+\dots+m_{k+1}=q}\mu_n\left[{\mathcal X^{m_1}}h_1{\mathcal X^{m_2}}\cdots h_k{\mathcal X^{m_{k+1}}}\right]\right\|<\frac\varepsilon8\|h_1\|\cdots\|h_k\|
$$
for any $r\ge q(\varepsilon,Q).$ Fix $r=2+q(\varepsilon,Q)$, and observe that by the norm-convergence of the moments of $\mu_n$ to the moments of $\mu$, we can find $N_\varepsilon\in\mathbb N$ so that for any $n\ge N_\varepsilon$ we have that whenever $\|h_1\|,\dots,\|h_k\|\le1$.
$$
\sum_{q=0}^r\sum_{m_1+\dots+m_{k+1}=q}\left\|\mu_n\left[{\mathcal X^{m_1}}h_1{\mathcal X^{m_2}}\cdots h_k{\mathcal X^{m_{k+1}}}\right]-\mu\left[{\mathcal X^{m_1}}h_1{\mathcal X^{m_2}}\cdots h_k{\mathcal X^{m_{k+1}}}\right]\right\|<\frac\varepsilon2
$$

This proves that $\lim_{n\to\infty}\|G_{\mu_n}^{(k}(Q)-G_\mu^{(k}(Q)\|=0$, as $k$-linear operators from $\mathcal B$ to $\mathcal D$. Since $k$ is arbitrary, the second condition of Theorem \ref{anal-conv} is satisfied, so $G_{\mu_n}$ converges locally uniformly in the norm topology of $\mathcal D$ to $G_\mu$, as claimed.

\end{proof}

The last two results have established the connection between transforms and distributions in the most general case that we will consider.

\begin{remark}\label{rmq8}

\item[(a)] Due to the relevance of this particular case, we emphasize again that if $\mathcal{B,D}$ are finite dimensional, then the sequence $\{\mu_n\}_{n\in\mathbb N}$ from Proposition \ref{continuity} needs not be uniformly bounded. Indeed, generally local compactness of finite dimensional spaces makes proofs considerably simpler. Unfortunately, infinite dimensional Banach spaces are not locally compact in the norm topology, so in particular closed balls are not compact. However, some properties of analytic maps on finite dimensional spaces remain true, including the continuity of composition operation \cite[Theorem 1.10]{isidro}.

\noindent\item[(b)] Using Theorem \ref{indiferent-conv} and \cite[Theorem 1.10]{isidro}, we can replace the Cauchy transform $G$ in Proposition \ref{continuity} with any of the transforms $F$, $B$, $M$, $\mathfrak H$, or, if $\mathcal B=\mathcal D$, $R$.
\item[(c)] Not the same thing can be said about the lemma preceding it; in that case, in the most general context $G$ can only be replaced by $\mathfrak H$ or $M$. However, there are many special cases in which $G$ can be replaced by $F$ or $R$.

\end{remark}

We introduce a few other new transforms and modifications of the transforms already introduced which will be of use to us. We recall the $\mathcal R$-transform, which can be expressed as $\mathcal R_\mu(b)=G_\mu^{-1}(b)-b^{-1}$ for $b$ invertible and of small norm, and which linearizes free additive convolution. (We use the convention that the $-1$ as exponent on the letter which denotes a function means compositional inverse, while a $-1$ exponent on the function evaluated in a point means the multiplicative inverse of the value of the function in that point: thus, $f(b)f(b)^{-1}=1$, while $f(f^{-1}(b))=b$.) From it we shall define the {\em Voiculescu transform $\varphi_\mu$ of $\mu$} simply as $\varphi_\mu(b)=\mathcal R_\mu(b^{-1})$. This function is easily seen to be defined on an open set in $\mathcal B$ and $F_\mu(\varphi_\mu(b)+b)=b$, so $\Im \varphi_\mu(b)\leq0$ whenever $\Im b>0.$ Obviously, the Voiculescu transform also satisfies
$$
\varphi_\mu(b)+\varphi_\nu(b)=\varphi_{\mu\boxplus\nu}(b).
$$

Finally, denoting $h_\mu(b)=\mathcal{F}_\mu(b)-b,$ $b\in\mathbb H^+(\ncspace{\cB})$, we re-write \eqref{B-transform} using \eqref{F-B}:
\begin{equation}\label{bool}
h_{\mu\uplus\nu}(b)=h_\mu(b)+h_\nu(b)\quad b\in \mathbb H^+(\ncspace{\cB}).
\end{equation}

We shall deal next with linearizing transforms for conditionally free convolutions. Distributions $(\mu,\nu)\in\Sbd\times\Sb$ can be associated a linearizing transform for c-free convolution, the $\CR$-transform. It is defined by the functional equation \begin{equation}\label{c-free}
[M_\mu(b)-\mathds{1}]\cdot M_\nu(b)=M_\mu(b)\cdot\CR_{\mu,\nu}(bM_\nu(b)).
\end{equation}
(We remind the reader that the second coordinate is linearized by the $R$-transform.)

 Note that when $\mathcal B=\mathcal D$ and $E_\mathcal B=\theta$ in Definition \ref{def-cfree} we obtain $\mu=\nu$, so c-free convolution simply coincides with free convolution for both coordinates. In addition, if $\nu$ is the distribution of the zero random variable (corresponding to $M_\nu(b)=\mathds{1}$), then we obtain $M_\mu(b)-\mathds{1}=M_\mu(b)\CR_{\mu,\nu}(b)$, which is equivalent to the  definition of the $B$-transform. Thus, conditionally free convolution can be viewed as interpolating between free and Boolean convolution, as in the scalar case \cite{bsk}.

We will often work in terms of selfadjoint random variables, and not elements in $\Sigma_{\mathcal B:\mathcal D}$ or $\Sigma_{\mathcal B:\mathcal D}^0$; of course, the two approaches are fully equivalent.

\subsection{Voiculescu's subordination result and c-freeness}\label{2.2}
Next, we come to the problem of subordination (the reason why we chose in one of the above definitions to write conditional expectation from $\mathcal A$ to $\mathcal B$ as $E_\mathcal B$ instead of $\varphi$). Generally we denote by $E_\mathcal V$ the conditional expectation from the ``large'' algebra onto the subalgebra $\mathcal V.$ Voiculescu shows in \cite[Theorem 3.8]{V2} that

\begin{thm}\label{subord}
Assume that the selfadjoint operator valued random variables $X$ and
$Y$ are free with amalgamation over $\mathcal B$. Then there exists a unique map $\omega^{(n)}\colon\mathbb H^+(M_n(\mathcal B))\to\mathbb H^+(M_n(\mathcal B))$ so that

\begin{equation}\label{B(X)}
E_{M_n(\mathcal B\langle X\rangle)}[(b-(X+Y)\otimes 1_n)^{-1}]=[\omega^{(n)}(b)-X\otimes1_n]^{-1},\quad b\in\mathbb H^+(M_n(\mathcal B)).
\end{equation}
In particular, $G_{X+Y}^{(n)}(b)=G_X^{(n)}(\omega^{(n)}(b))$ for all $n\in\mathbb N$, $b\in \mathbb H^+(M_n(\mathcal B))$. In addition, $\omega=(\omega^{(n)})_n$ is a non-commutative function and $\Im\omega(b)\ge\Im b$ for all $b\in \mathbb H^+(\ncspace{\cB})$.
\end{thm}

Now let us look at the subordination problem from two different perspectives. First, pick $\mathcal{ A, B, D}$, $\theta,E_\cB$ as in Definition \ref{def-cfree}. Let $X,Y$ be selfadjoint and c-free in $\mathcal A$ over the pair of algebras $\mathcal{B, D}$. We call their (pair) distributions $(\mu_X,\nu_X)$ and $(\mu_Y,\nu_Y)$. We shall re-express the c-free convolution of those two in terms of the subordination functions for the second coordinates. We note first that relation \eqref{c-free} holds for $b$ of small enough norm. Next, recall that if we require in addition that $b$ is also invertible, then $M_\mu(b)=b^{-1}\mathcal G_\mu(b^{-1})$. This holds in general, for distributions $\mu\in\Sigma_{\mathcal{B:D}}$. In such a case, $G^{(n)}_{\mu_X}(b)=(\theta\otimes 1_n)[(b-X)^{-1}]$ for $b$ with positive imaginary part, or with norm of its inverse small enough. 
Now we rewrite \eqref{c-free} in terms of $\mathcal{G}$:
$$
(b^{-1}\mathcal{G}_\mu(b^{-1})-\mathds{1})b^{-1}\mathcal{G}_\nu(b^{-1})=
b^{-1}\mathcal{G}_\mu(b^{-1})\CR_{\mu,\nu}(\mathcal{G}_\nu(b^{-1})).
$$
Replace $b^{-1}$ by $b$:
$$
(b\cdot\mathcal{G}_\mu(b)-1)\cdot b\cdot\mathcal{G}_\nu(b)=
b\cdot\mathcal{G}_\mu(b)\CR_{\mu,\nu}(\mathcal{G}_\nu(b)).
$$
Observe that here we can simplify a $b$ to get
$$
\mathcal{G}_\mu(b)\cdot b\cdot\mathcal{G}_\nu(b)-\mathcal{G}_\nu(b)=
\mathcal{G}_\mu(b)\CR_{\mu,\nu}(\mathcal{G}_\nu(b)).
$$
Now take $(\mu,\nu)=(\mu_X,\nu_X)$, denote by $\omega_1$ the subordination function with respect to $X$ for the second coordinate (so that $E_{\mathcal B\langle X\rangle}[(b-X-Y)^{-1}]=[\omega_1(b)-X]^{-1}$) and replace $b$ by $\omega_1(b)$. Since $\Im\omega_1(b)\ge\Im b$, the equation will hold provided $(\Im b)^{-1}$ is sufficiently small. We get
$$
\mathcal{G}_{\mu_X}(\omega_1(b))\omega_1(b)\mathcal{G}_{\nu_X}(\omega_1(b))-\mathcal{G}_{\nu_X}(\omega_1(b))=\mathcal{G}_{\mu_X}(\omega_1(b))\CR_{\mu_X,\nu_X}(\mathcal{G}_{\nu_X}(\omega_1(b))).
$$
Finally, we multiply to the left by $\mathcal{F}_{\mu_X}(\omega_1(b))$:
$$
\omega_1(b)\mathcal{G}_{\nu_X}(\omega_1(b))-\mathcal{F}_{\mu_X}(\omega_1(b))\mathcal{G}_{\nu_X}(\omega_1(b))
=\CR_{\mu_X,\nu_X}(\mathcal{G}_{\nu_X}(\omega_1(b))).
$$
Repeat the same process with $X$ replaced by $Y$ and $\omega_1$ by $\omega_2$ to get
$$
\omega_2(b)\mathcal{G}_{\nu_Y}(\omega_2(b))-\mathcal{F}_{\mu_Y}(\omega_2(b))\mathcal{G}_{\nu_Y}(\omega_2(b))
=\CR_{\mu_Y,\nu_Y}(\mathcal{G}_{\nu_Y}(\omega_2(b))).
$$
Two remarks which are immediate: first, from Theorem \ref{subord}, we have $\mathcal{G}_{\nu_X}(\omega_1(b))=\mathcal{G}_{\nu_Y}(\omega_2(b))=\mathcal{G}_{\nu_{X+Y}}(b).$ Second, $\CR$ linearizes c-free convolution. Thus, replacing in the above relations and adding them gives us
\begin{eqnarray*}
\lefteqn{\CR_{\mu_{X+Y},\nu_{X+Y}}(\mathcal{G}_{\nu_{Y+X}}(b))=}\\
& & \mbox{}\left[\omega_1(b)-\mathcal{F}_{\mu_X}(\omega_1(b))+\omega_2(b)-\mathcal{F}_{\mu_Y}(\omega_2(b))\right]\mathcal{G}_{\nu_{X+Y}}(b)
\end{eqnarray*}
But this means (if we express $\CR$ properly) that
\begin{eqnarray*}
\lefteqn{(b-\mathcal{F}_{\mu_{X+Y}}(b))\mathcal{G}_{\nu_{X+Y}}(b)=}\\
& & \mbox{}\left[\omega_1(b)-\mathcal{F}_{\mu_X}(\omega_1(b))+\omega_2(b)-\mathcal{F}_{\mu_Y}(\omega_2(b))\right]\mathcal{G}_{\nu_{X+Y}}(b)
\end{eqnarray*}
Denoting $h(b)=\mathcal{F}(b)-b$ and simplifying the invertible
$\mathcal{G}_{\nu_{X+Y}}(b)$ provides us with the operator-valued analogue of \cite[Corollary 4]{B-proc}:
\begin{equation}\label{3}
h_{\mu_{X+Y}}(b)=h_{\mu_{X}}(\omega_1(b))+h_{\mu_{Y}}(\omega_2(b)),
\quad b\in\mathbb H^+(\mathcal B).
\end{equation}
If there is one drawback to this formula it is that we cannot state that $\Im F_\mu(b)\ge\Im b$, as the target algebra for $F$ is $\mathcal D$ (or, to be more precise, its upper half-plane.) Let us also note that $B_{\mu_X}(b^{-1})b=b-F_{\mu_X}(b)=-h_{\mu_X}(b),$ so the above can be written (however in a less pleasant form) in terms of $B$.

For our purposes (related to the infinite divisibility and the triangular arrays of identically distributed rows), we note that the $n$ times c-free convolution is given as
\begin{equation}\label{13}
h_{\mu_{X_1+\cdots+X_n}}(b)=nh_{\mu_{X_1}}(\omega_n(b)),
\end{equation}
where $\omega_n$ is the subordination function with respect to n-fold free additive convolution: $\mathcal{G}_{\nu_{X_1+\cdots+X_n}}(b)=\mathcal{G}_{\nu_{X_1}}(\omega_n(b))$.

In addition to this result, we would like to emphasize the connection of c-free independence with Markovianity, as discussed by Voiculescu. Let us recall that $\mu_X\colon\mathcal B\langle\mathcal X\rangle\to\mathcal D$ is given by $\mu_X(P(\mathcal X))=\theta(P(X))\in\mathcal D$. In particular, we can look at Voiculescu's subordination theorem and apply $\theta$ to its formula:
$$
\theta\left(E_{\mathcal B\langle X\rangle}[(b-(X+Y))^{-1}]\right)=\theta\left([\omega_1(b)-X]^{-1}\right),\quad b\in\mathbb H^+(\cB)
$$
Since $\omega_1(b)\in\mathbb H^+(\mathcal B)$, it is clear from the definition that $\theta\left([\omega_1(b)-X]^{-1}\right)=G_{\mu_X}(\omega_1(b))$. In particular, if $\mathcal D$ happens to be any von Neumann algebra so that $\mathcal B\subset\mathcal D\subset\mathcal B\langle X\rangle$ and $\theta$ itself is a conditional expectation, this simply indicates that c-freeness is in fact a different expression of Markovianity, or, differently said, c-freeness generalizes the Markov property for free algebras.


We conclude this section with a short remark: it follows from \eqref{3} with $\mu_X=\nu_X$, $\mu_Y=\nu_Y$ that $h_{\nu_X}(\omega_1(b))+h_{\nu_Y}(\omega_2(b))=h_{\nu_{X+Y}}(b)$, so
\begin{equation}\label{B-V}
\mathcal{F}_{\nu_{X+Y}}(b)=\omega_1(b)+\omega_2(b)-b\quad b\in\mathbb H^+(\ncspace{B}).
\end{equation}


\section{A Hin\v{c}in type theorem for free additive convolution and the restricted Boolean Bercovici-Pata bijection}
We shall first prove a restricted version of the Boolean Bercovici-Pata bijection, using methods inspired by \cite{BN}. There the bijection was conveniently expressed in terms of the subordination function for the power two free additive convolution. We shall obtain the same result here: it follows from the operator-valued $R$-transform property
\eqref{R-transform}, 
Theorem \ref{subord} and analytic continuation that
\begin{prop}\label{5}
For any $\mathcal B$-valued distribution $\mu$, we
denote $\omega$ the subordination function for
$\mu\boxplus\mu$.
Then the following
functional equations hold {\em:
\begin{equation}\label{o}
\omega(b)=\frac12(b+F_{\mu\boxplus\mu}(b))=\frac12(b+F_\mu
(\omega(b)),
\end{equation}
\begin{equation}\label{F}
F_{\mu\boxplus\mu}(b)=F_\mu\left(\frac12(b+F_{\mu\boxplus\mu}(b))
\right),\quad b\in \mathbb H^+(\mathcal B).
\end{equation}
}
Both these equations hold for the fully matricial extensions of the functions involved.
\end{prop}

We prove next a Hin\v{c}in type theorem.
\begin{thm}\label{hincin}
Assume that $\{X_{jk}\}_{j\in\mathbb N,1\leq k\leq k_j}$ is a triangular array of random variables in 
$(\mathcal A,E_\mathcal B,\mathcal B)$ of elements free over $\mathcal B$ so that
$\{X_{jk}:1\le k\le k_j\}$ have the same distribution with respect to
$E_\mathcal B$ for each $j\in\mathbb N$ (i.e. rows are identically distributed). Assume in addition that
$$
\limsup_{j\to\infty}
\|X_{j1}+\cdots+X_{jk_j}\|
\leq M
$$
for some $M\ge0$.
If $\lim_{j\to\infty}X_{j1}+X_{j2}+\cdots+X_{jk_j}$ exists
in distribution as norm-limit of moments in $\Sigma_\mathcal B$, then the limit distribution is freely infinitely divisible over 
$\cB$.
\end{thm}

\begin{remark/definition}
We shall call a triangular array of identically distributed rows satisfying $\lim_{j\to\infty}k_j=\infty$ and $\limsup_{j\to\infty}
\|X_{j1}+\cdots+X_{jk_j}\|\leq M$ for some $M\ge0$ {\em infinitesimal}. We shall call a triangular array of distributions infinitesimal if they can be realized operatorially by an infinitesimal triangular array of random variables. It should be noted that in scalar-valued probability, being infinitesimal means that for any $\epsilon>0$, $\lim_{j\to\infty}\int\chi_{[-\epsilon,\epsilon]}(t)\,d\mu_{X_{j1}}(t)=1$; thus, we require a stronger notion of infinitesimality for our theorem.
It should be noted again, however, that when $\mathcal B$ is finitely dimensional, the above theorem remains true even when the stronger requirement of infinitesimality is removed.
\end{remark/definition}

\begin{proof} We first observe that, due to the fact that $\mathcal G_X$ is a noncommutative function and the upper half-plane an admissible noncommutative set, it is enough to consider in our proof the functions $G_X$.
Denote $\mu_j$ the distribution of $X_{1j}$, i.e.
$G_{\mu_j}=E_\mathcal B[(b-X_{1j})^{-1}]$.
We shall use the subordination result:
$$
G_{\mu_j}(\omega_j(b))=G_{\mu_j^{\boxplus k_j}}(b),\quad\omega_j(b)=
\frac1{k_j}b+\left(1-\frac1{k_j}\right)F_{\mu_j}(\omega_j(b)),
\quad b\in\mathbb H^+(\mathcal B).$$
we observe that $\omega_j$ is in fact the
reciprocal of the Cauchy transform of a positive $\mathcal B$-valued
distribution; indeed, it is clear from \eqref{bool} that $\omega_j$ is indeed the
reciprocal of the Cauchy transform of $(\mu_j^{\boxplus k_j})^{\uplus 1-k_j^{-1}}$. Using the characterization in terms of $R$-transform of the free infinite divisibility of Popa and Vinnikov \cite[Theorem 5.9]{popa-vinnikov}, as the
operator-valued Voiculescu transform of this distribution
is simply $\varphi(w)=(k_j-1)(w-F_{\mu_j}(w)),$ $w\in\mathbb H^+(\mathcal B)$, it follows that $(\mu_j^{\boxplus k_j})^{\uplus 1-k_j^{-1}}$ is freely infinitely divisible.

Now our proof is complete: the subordination relation tells us that\\
$\lim_{j\to\infty}\omega_j(b)=\lim_{j\to\infty}F_{\mu^{\boxplus k_j}_j}(b)$ norm-uniformly
on subsets $D\subset\subset\mathbb H^+(\mathcal B)$, so
$\lim_{j\to\infty}(\mu_j^{\boxplus k_j})^{\uplus 1-k_j^{-1}}=
\lim_{j\to\infty}\mu_j^{\boxplus k_j}.$ Since, as noted above,
$(\mu_j^{\boxplus k_j})^{\uplus 1-k_j^{-1}}$ is freely infinitely
divisible and the set of freely infinitely divisible operator-valued
distributions is closed under taking norm-moment limits, by Proposition \ref{continuity} we are done.
\end{proof}

\begin{thm}\label{boolean}
Consider two infinitesimal triangular arrays $\{X_{jk}\}_{j\in\mathbb N
,1\leq k\leq k_j}$ and $\{Y_{jk}\}_{j\in\mathbb N
,1\leq k\leq k_j}$ in $(\mathcal A,E_\mathcal B,\mathcal B)$ so that $X_{jk}$ are all free with amalgamation over $\mathcal B$, $Y_{jk}$ are Boolean independent with amalgamation over $\mathcal B$, and $\{X_{jk}:1\le k\le k_j\}\cup\{Y_{jk}\colon 1\le k\le k_j\}$ have the same distribution with respect to $E_\mathcal B$ for each $j\in\mathbb N$ (i.e. rows are identically distributed).
The following are equivalent:
\begin{enumerate}
\item[(a)] $\lim_{j\to\infty}X_{j1}+X_{j2}+\cdots+X_{jk_j}$ exists
in distribution (as norm-limit convergence of moments); we call the limit distribution $\mu_X$;
\item[(b)] $\lim_{j\to\infty}Y_{j1}+Y_{j2}+\cdots+Y_{jk_j}$ exists
in distribution (as norm-limit convergence of moments); we call the limit distribution $\mu_Y$.
\end{enumerate}
Moreover, the correspondence between the two limiting distributions
is given analytically by the fully matricial extension of the relation
\begin{equation}\label{bbp}
F_{\mu_X}(b)=\frac12\left(b+F_{\mu_Y}(F_{\mu_X}(b))\right)\quad b\in
\mathbb H^+(\mathcal B).
\end{equation}
\end{thm}
The correspondence $\mu_X\leftrightarrow\mu_Y$ is the Boolean
Bercovici-Pata bijection. It can be easily seen to be, as in the case
of scalar-valued distributions, a morphism between $\boxplus$ and
$\uplus$.
\begin{proof}
Let us assume that
$\lim_{j\to\infty}X_{j1}+X_{j2}+\cdots+X_{jk_j}$ exists. As seen
in Theorem \ref{hincin}, $\mu_X$ is $\boxplus$-infinitely
divisible.
In terms of the subordination function, this simply corresponds to
the limit described in the proof of the previous theorem.

On the other hand, if existing,
$F_Y(b)-b=\lim_{j\to\infty}k_j(F_{Y_{j1}}(b)-b)$.
Since we know that each of  $k_j(b-F_{Y_{j1}}(b))$ is itself a Voiculescu transform of
a probability measure
(namely of $(\mu_{X_{j1}+\cdots+X_{jk_j}+X_{jk_j+1}})^{\uplus 1-\frac{1}{k_j+1}}=
(\mu_{X_{j1}}^{\boxplus k_j+1})^{\uplus 1-\frac{1}{k_j+1}}$), it is enough to show that the inverse of
$b+k_j(F_{Y_{j1}}(b)-b)$
with respect to composition (which is $\boxplus$-infinitely divisible), converges to
$F_X$ uniformly on compacts.
But this inverse is simply $\omega_{j}$ from the previous theorem's proof.
This completes the proof of one implication.

The converse is simpler. The statement that
$Y_{j1}+\cdots+Y_{jk_j}$ tends to $Y$ in distribution as $j\to\infty$
is equivalent to $\lim_{j\to\infty}k_j(F_{Y_{j1}}(b)-b)=F_Y(b)-b$ uniformly
on compacts of $\mathbb H^+(\mathcal B)$. But then the expression of the Voiculescu transform of the distribution associated with
$\omega_j$ is simply $\varphi_j(w)=(k_j-1)(w-F_{Y_{j1}}(w))$, $w\in\mathbb H^+(\mathcal B)$.
Since convergence of Voiculescu transforms and convergence in
distribution are equivalent, it follows that the distribution associated to
$\omega_j$ converges weakly. Since, as seen in the proof of the previous theorem,
the distribution associated to $\omega_j$ is simply the $(1-k_j^{-1})$th
Boolean power of the distribution of $X_{j1}+\cdots+X_{jk_j}$, it follows that $X_{j1}+\cdots+X_{jk_j}$
also converges in distribution to the same limit. This shows that $\lim_{j\to\infty}
X_{j1}+\cdots+X_{jk_j}$ exists in distribution and its reciprocal Cauchy transform's formula is the one
indicated in \eqref{bbp}.
\end{proof}
As a by-product, we obtain the following corollary describing distributions which are the second power with respect to free additive convolution.
\begin{cor}
An operator-valued distribution $\mu$ is the $n^{\rm th}$ power with respect to free additive convolution of an operator-valued distribution
if and only if $\mu^{\uplus 1-1/n}$ is freely infinitely divisible.
\end{cor}
\begin{proof}
First implication has been noted in the proof of Theorem \ref{hincin}, where it is noted that $(\nu^{\boxplus n})^{\uplus 1-1/n}$ must be freely infinitely divisible. The converse follows from \cite{popa-vinnikov}: if $\mu$ is freely infinitely divisible, then as $\mu^{\uplus n/(n-1)}$ satisfies $\mathcal{F}_{\mu^{\uplus n/(n-1)}}(b)=\frac{n}{n-1}\mathcal{F}_\mu(b)-\frac{1}{n-1}b,$ we can use the definition of $\varphi_\mu$ to conclude $\mathcal{F}_{\mu^{\uplus n/(n-1)}}(b+\varphi_\mu(b))=b-\frac{1}{n-1}\varphi_\mu(b)$. Applying $\mathcal{F}_{\mu^{\uplus n/(n-1)}}^{-1}$ and using again the definition of $\varphi$ gives $\varphi_\mu(b)=\varphi_{\mu^{\uplus n/(n-1)}}(b-\frac{1}{n-1}\varphi_\mu(b))-\frac{1}{n-1}\varphi_\mu(b).$ Simple arithmetic gives $\frac1n\varphi_{\mu^{\uplus n/(n-1)}}(b-\frac1{n-1}\varphi_\mu(b))+b-\frac1{n-1}\varphi_\mu(b)=b$. All these relations are valid whenever $\Im b$ has an inverse which is small enough. Analytic continuation gives $\mathcal{F}_{\left(\mu^{\uplus n/(n-1)}\right)^\frac1n}(b)=b-\frac{1}{n-1}\varphi_\mu(b)$. The description from \cite{popa-vinnikov} of freely infinitely divisible distributions in terms of their $R$-transform together with the relation between $R_\mu$ and $\varphi_\mu$ guarantees that $b-\frac{1}{n-1}\varphi_\mu(b)$ is well-defined on all $\mathbb H^+(\ncspace{\cB})$ with imaginary part at least $\Im b$.
\end{proof}

\section{C-free Bercovici-Pata bijection}

In this section we shall connect Boolean, free and conditionally free infinitely divisible distributions both via their Hin\v{c}in-type description as limits of triangular arrays and via explicit formulas linking their transforms. Our results will generalize the results of the previous section, but we will also use those results in our proofs.

In some of the statements and proofs made below, we will use the noncommutative set of nilpotent elements of our C${}^\ast$-algebras and noncommutative functions defined on it; we are aware that there exist many C${}^\ast$-algebras that have no nilpotent elements except zero, however the fully matricial extensions of those elements (as noted in Section 2) are very rich. As we did before, we sometimes write the proofs using the first coordinate of our noncommutative maps and sets, but we do this only for the relative simplicity of notation: the noncommutative structure allows each proof to be re-written for any coordinate.

\begin{lemma}
 Let $\{k_n\}_n$ be an increasing sequence of positive integers and $\{\mu_n\}_n$ be a sequence of elements from $\Sbd$.
 Then $\mu_n^{\uplus k_n}$ norm-converges in moments to $\mu\in\Sbd$  if and only if for all $b\in\Nilp(\cB)$
 \begin{equation}\label{eq71}
  B_\mu(b)=\lim_{n\to\infty} k_n\cdot[ M_{\mu_n}(b)-\mathds{1}].
 \end{equation}
\end{lemma}
\begin{proof}
 Since for all $b\in\Nilp(\cB)$ we have that $M_\mu(b)-\mathds{1}\in\Nilp(\cD)$ as noted in Section 2, it follows that $M_\mu(b)$ is invertible, so
 \[
 B_\nu(b)=[M_\nu(b)-\mathds{1}]\cdot M_\nu(b)^{-1}
 \]
 for any $\nu\in\Sbd$, hence the norm-convergence in moments of $\mu_n^{\uplus k_n}$ is equivalent to
  \begin{equation}\label{eq:folos2}
 \lim_{n\to\infty}k_nB_{\mu_n}(b)=B_\mu(b).\ \text{for all}\ b\in\Nilp(\cB)
 \end{equation}

 Suppose first that (\ref{eq:folos2}) holds true. Then $\lim_{n\to\infty}B_{\mu_n}(b)=\frac{1}{k_n}B_\mu(b)=0$, but $[\mathds{1}-B_{\mu_n}(b)]\cdot M_{\mu_n}(b)=\mathds{1}$, therefore $\lim_{n\to\infty}M_{\mu_n}(b)^{-1}=\mathds{1}$, so
 \begin{eqnarray*}
 \lim_{n\to\infty} k_n[M_{\mu_n}(b)-\mathds{1}]
 &=&
 \lim_{n\to\infty} k_n[M_{\mu_n}(b)-\mathds{1}]M_{\mu_n}(b)^{-1}\\
 &=&
 \lim_{n\to\infty} k_n B_{\mu_n}(b)\\
 &=&B_\mu(b).
 \end{eqnarray*}

 For the converse, if (\ref{eq71}) holds true, then we have $\lim_{n\to\infty}(M_{\mu_n}(b)-\mathds{1})=\lim_{n\to\infty}\frac{1}{k_n}B_{\mu}(b)=0$, that is $\lim_{n\to\infty}M_{\mu_n}(b)^{-1}=\mathds{1}$, therefore
 \begin{eqnarray*}
 \lim_{n\to\infty}k_n[M_{\mu_n}(b)-\mathds{1}]
 &=&
 \lim_{n\to\infty}k_n[M_{\mu_n}(b)-\mathds{1}]M_{\mu_n}(b)^{-1}\\
 &=&\lim_{n\to\infty}k_n B_{\mu_n}(b).
 \end{eqnarray*}
\end{proof}

 The next lemma gives a similar characterization for the linearizing transforms of c-free convolution.
\begin{lemma}\label{lema19}
  Let $\{k_n\}_n$ be an increasing sequence of positive integers and $(\mu,\nu)\in\Sbd^0\times\Sb^0$, respectively $\{(\mu_n, \nu_n)\}_n$, be 
an infinitesimal sequence from $\Sbd\times\Sb$. Then the following are equivalent:
 \begin{enumerate}
 \item[a.]
 $(\mu_n, \nu_n)^{\cfree k_n}$ norm-converges in moments to $(\mu,\nu)$
 \item[b.] for all $b\in\Nilp(\cB)$ we have that
 \begin{eqnarray*}
 R_{\nu}(b)&=&\lim_{n\lra\infty} k_n\cdot M_\nu(b)\\
 {}^cR_{\mu,\nu}(b)&=&\lim_{n\lra\infty} k_n\cdot\ M_\mu(b).
 \end{eqnarray*}
 \end{enumerate}
\end{lemma}
\begin{proof}
The implication``(b)$\Rightarrow$(a)'' is a re-phrasing of Theorem \ref{boolean} and the proof is identical. We give it here for the convenience of the reader. Recall the equation $F_{\nu}^{-1}(b)-b=\mathcal \varphi_{\nu}(b)$, for $b\in\mathbb H^+(\mathcal B)$ with $\|b^{-1}\|$ small. The convergence for the second coordinate comes to $\lim_{n\to\infty}k_n\varphi_{\nu_n}(b)=\varphi_{\nu}(b).$ We shall (for now formally) replace $b$ by $F_{\nu_n}(b)$. This gives us
$$
\varphi_{\nu}(F_{\nu_n}(b))=\lim_{n\to\infty}k_n\varphi_{\nu_n}(F_{\nu_n}(b))=\lim_{n\to\infty}k_n(b-F_{\nu_n}(b))=\lim_{n\to\infty}k_n
B_{\nu_n}(b^{-1})b.
$$
We note that indeed we are indeed allowed to make the substitution and take limits by \cite[Theorem 1.10]{isidro}, as $F_{\nu_n}(b)\to b$ in norm, uniformly on closed balls inside any proper region of the upper half-plane, by Remark \ref{rmq8}. same argument works for fully matricial extensions of the above maps. This together with the previous lemma and the equivalence of convergence on $\Nilp$ and $\mathbb H^+$ proves the first statement.

The implication ``(a)$\Rightarrow$(b)'' is now direct, from the definition of $ {}^cR_{\mu,\nu}(b)$ as given in \eqref{c-free}:
$$
(M_\mu(b)-\mathds{1})\cdot M_\nu(b)
=M_\mu(b)
\CR_{\mu,\nu}(bM_\nu(b)
).
$$
We re-write this as $k_n\cdot B_{\mu_n}(b)b^{-1}\mathfrak{H}_{\nu_n}(b)=k_n\cdot \CR_{\mu_n,\nu_n}(\mathfrak H_{\nu_n}(b))$ and note again that $\mathfrak H_{\nu_n}(b)$ converges uniformly to $b$. Using the previous lemma, the linearizing property of $\CR$ and Remark \ref{rmq8}, we conclude.
\end{proof}

 As shown in \cite{popa-vinnikov} for each $\mu\in\Sbd$, there exist some selfadjoint $\gamma\in\cB$ and some linear map $\sigma:\bx\lra\cD$ satisfying \eqref{prop} such that
 \begin{equation}\label{eq72}
 B_\mu(b)=\left[\gamma\cdot\mathds{1} +\widetilde{\sigma}\left(b(\mathds{1}-\X b)^{-1}\right)\right]\cdot b.
 \end{equation}
where $\widetilde{\sigma}$ is the fully matricial extension of $\sigma$ to $\cB\langle\langle\X\rangle\rangle$ - the formal non-commutative power series with coefficients in $\cB$.

 Also, if $(\mu,\nu)\in\Sbd\times\Sb$ is $\cfree$-infinitely divisible, then there exist some selfadjoint $\gamma_0\in\cB$, $\gamma_1\in\cD$ and some linear maps  $\sigma_0:\bx\lra\cB$, $\sigma_1:\bx\lra\Sbd$ satisfying \ref{prop} such that
 \begin{eqnarray}
 R_\nu(b)=\left[\gamma_0\cdot\mathds{1} +\widetilde{\sigma_0}\left(b(\mathds{1}-\X b)^{-1}\right)\right]\cdot b\label{eq73}\\
 {}^cR_{\mu, \nu}(b)=\left[\gamma_1\cdot\mathds{1} +\widetilde{\sigma_1}\left(b(\mathds{1}-\X b)^{-1}\right)\right]\cdot b\label{eq74}
 \end{eqnarray}
 Moreover, if the moments of $\mu,\nu$ do not grow faster than exponentially (that is $(\mu,\nu)\in\Sbd^0\times\Sb^0$) then so do the moments of $\sigma, \sigma_0$ and $\sigma_1$ from above.

 \begin{defn}\label{def43}

 We define the bijection $\mathcal{BP}$ from $\Sbd\times\Sb$ to the set of all $\cfree$-infinitely divisible elements of $\Sbd\times\Sb$ as follows: if $B_\mu$, respectively $B_\nu$ are determined as above by the pairs $(\gamma, \sigma)$ and $(\gamma_0, \sigma_0)$, then $\mathcal{BP}(\mu,\nu)$ is the $\cfree$-infinitely divisible pair $(\mu^\prime, \nu^\prime)$ such that ${}^cR_{\mu^\prime, \nu^\prime}$ and $R_{\nu^\prime}$ are determined, as above, by $(\gamma, \sigma)$, respectively $(\gamma_0, \sigma_0)$.

 \end{defn}

\begin{remark}\label{remark44}
$\mathcal{BP}$ maps $\Sbd^0\times\Sb^0$ onto the $\cfree$-infinitely divisible elements of $\Sbd^0\times\Sb^0$. Moreover, if $(\mu,\nu)\in\Sbd^0\times\Sb^0$ and $\mu^\prime$ is the first coordinate of $\mathcal{BP}(\mu,\nu)$, then
\begin{eqnarray}\label{c-free-bbp}
h_{\mu^\prime}(b) & = & h_\mu(F_{\mathcal{BP}(\nu)}(b))\nonumber\\
\mathcal{F}_{\mathcal{BP}(\nu)}(b) & = & \frac12(b+\mathcal{F}_\nu(\mathcal{F}_{\mathcal{BP}(\nu)}(b))),\quad b\in\mathbb H^+(\ncspace{\cB}).
\end{eqnarray}
\end{remark}
\begin{proof}

 The first assertion is an immediate consequence of Remark 5.5 from \cite{popa-vinnikov}; more precisely if the components of the $R$- and ${}^cR$-transforms of $\nu$ and $(\mu,\nu)$ do not grow faster than exponentially, then so do the moments of $\nu$ and $\mu$.

 The second equation of the second assertion is simply equation \eqref{bbp}; looking at $\mathcal R_{\mathcal{BP}(\nu)}(b)=\gamma_0\cdot\mathds{1}
+\widetilde{\sigma_0}\left((b^{-1}-\X)^{-1}\right)$, $\Im b<0$, we observe that $\varphi_{\mathcal{BP}(\nu)}(b)=\mathcal R_{\mathcal{BP}(\nu)}(b^{-1})=\gamma_0\cdot\mathds{1}
+\widetilde{\sigma_0}\left((b-\X)^{-1}\right)$, $\Im b>0$. We note that if $\Im b^{-1}$ is small in norm, then $\Im\widetilde{\sigma_0}\left((b-\X)^{-1}\right)$ is also small in norm; indeed, this has been noted in the proof of Proposition \ref{continuity}. Thus, when $\Im b^{-1}$ is small enough, $b+\varphi_{\mathcal{BP}(\nu)}(b)\in\mathbb H^+(\ncspace{\cB})$.  Using the expression of $\varphi$ in terms of $\mathcal{F}$, by replacing $b$ with $b+\varphi_{\mathcal{BP}(\nu)}(b)$ our equation is equivalent to
\begin{eqnarray*}
b & = & \frac{1}{2}\left(b+\varphi_{\mathcal{BP}(\nu)}(b)+\mathcal{F}_\nu(b)\right),
\end{eqnarray*}
which is trivially equivalent to
\begin{eqnarray*}
\gamma_0\cdot\mathds{1}
+\widetilde{\sigma_0}\left((b-\X)^{-1}\right) & = & -\varphi_{\mathcal{BP}(\nu)}(b)\\
& = & \mathcal{F}_\nu(b)-b\\
& = & -B_\nu(b^{-1})b.
\end{eqnarray*}
Analytic continuation concludes the proof of the second equation.

We note once again that this second equation together with Proposition \ref{5} indicates that $\mathcal{F}_{\mathcal{BP}(\nu)}$ is in fact simply the {\em subordination function} corresponding to the free additive convolution of $\nu$ with itself: $\mathcal{F}_\nu(\mathcal{F}_{\mathcal{BP}(\nu)}(b))=\mathcal{F}_{\nu\boxplus\nu}(b)$. For the first equation let us simply rewrite the defining relation \eqref{c-free} for $\CR$ expressed in terms of $\gamma_1$ and $\widetilde{\sigma_1}$ as
$$
-h_{\mu'}(b)=b-\mathcal{F}_{\mu'}(b)=\gamma_1\cdot\mathds{1}+\widetilde{\sigma_1}\left((\mathcal{F}_{\mathcal{BP}(\nu)}(b)- \X)^{-1}\right),\quad b\in\mathbb H^+(\ncspace{\cB}).
$$
As in the definition of the Bercovici-Pata bijection $B_\mu$ is given by $\gamma_1$ and $\widetilde{\sigma_1}$, this is the claimed relation.
\end{proof}

We go now to the main result of this section.
\begin{thm}

Let $(\mu_n,\nu_n)_n$ be an infinitesimal sequence from $\Sbd\times\Sb$ and $\{k_n\}_n$ be an increasing sequence of positive integers. The following properties are equivalent:
\begin{enumerate}
\item[(1)] $(\mu_n, \nu_n)^{\uplus k_n}$ norm-converges in moments to some $(\mu, \nu)\in\Sbd^0\times\Sb^0$
\item[(2)]$(\mu_n, \nu_n)^{\cfree\, k_n}$ norm-converges in moments in $\Sbd^0\times\Sb^0$
\item[(3)]$(\mu_n, \nu_n)^{\cfree\, k_n}$ norm-converges in moments to $\mathcal{BP}(\mu, \nu)$ for some $(\mu, \nu)\in\Sbd^0\times\Sb^0$
\item[(4)] There exist the pairs $(\gamma, \sigma)$ and $(\gamma_0, \sigma_0)$ determining $B_\mu$ and $B_\nu$ as above such that for all $m$ amd all $b_1,\dots,b_m\in\cB$:
\begin{align*}
\lim_{n\lra\infty}&(\mu_n)(\X)=\gamma\\
\lim_{n\lra\infty}&(\nu_n)(\X)=\gamma_0\\
\lim_{n\lra\infty}&k_n\cdot \mu_n(\X b_1\X b_2\cdots b_m\X)=\sigma(b_1\X b_2\cdots\X b_m)\\
\lim_{n\lra\infty}&k_n\cdot\nu_n(\X b_1\X b_2\cdots b_m\X)=\sigma_0(b_1\X b_2\cdots\X b_m)
\end{align*}

\end{enumerate}
\end{thm}
\begin{proof} Note that, by Theorem \ref{boolean} $\nu_n^{\uplus k_n}$ converges if and only if $\nu_n^{\boxplus k_n}$ converges. The two limits will be called in this proof $\nu$ and $\nu_\boxplus$ respectively. We observe that by the previous remark and Theorem \ref{boolean}, $\nu_\boxplus=\mathcal{BP}(\nu).$ As shown in subsection \ref{2.2}, $(\mu_{X_{1,n}+\cdots+X_{k_n,n}},\nu_{X_{1,n}+\cdots+X_{k_n,n}})=(\mu_n, \nu_n)^{\cfree k_n}$ can be expressed coordinatewise in terms of transforms:
$$
\begin{array}{cc}
h_{\mu_{X_{1,n}+\cdots+X_{k_n,n}}}(b)=k_nh_{\mu_{X_{1,n}}}(\omega_n(b)) & (\textrm{c-free})\\
 & \\
\omega_n(b)=\frac1{k_n}b+\left(1-\frac1{k_n}\right)F_{\nu_{X_{1,n}+\cdots+X_{k_n,n}}}(b)=\frac1{k_n}b+\left(1-\frac1{k_n}\right)\mathcal{F}_{\nu_{X_{1,n}}}(\omega_n(b)) & (\textrm{free})
\end{array}
$$
Recall from the proof of Theorem \ref{boolean} that $\omega_n=\mathcal{F}_{\nu_{X_{1,n}+\cdots+X_{k_n,n}}^{\uplus1-k_n^{-1}}}$ and $\nu_{X_{1,n}+\cdots+X_{k_n,n}}^{\uplus1-k_n^{-1}}$ converges to $\nu_\boxplus$.

$(1)\implies(2)$: Assume first that (1) holds. Then, as seen above, $\nu_n^{\boxplus k_n}$ converges to $\nu_\boxplus$. Since $\mu_n^{\uplus k_n}$ converges to $\mu$, it was shown in Proposition \ref{continuity} that $\lim_{n\to\infty}k_nh_{\mu_{X_{1,n}}}(b)=h_{\mu}(b)$ uniformly on closed bounded subsets of the upper half-plane which are at positive distance from $\partial\mathbb H^+(\mathcal B)$. Thus,
$$
\lim_{n\to\infty}k_n\cdot h_{\mu_{X_{1,n}}}(\omega_n(b))=h_\mu(F_{\nu_\boxplus}(b)),\quad b\in\mathbb H^+(\mathcal B),
$$
and the limit is uniform on closed bounded subsets of the upper half-plane, as shown in \cite[Theorem 1.10]{isidro}. Same argument works for the fully matricial extensions, and we conclude that (2) holds.

$(2)\implies(1)$: Next, assume that (2) holds. This clearly means that in relation (free) above we can take limits when $n$ tends to infinity to obtain that both $\nu_{X_{1,n}+\cdots+X_{k_n,n}}^{\uplus1-k_n^{-1}}$ and $\nu_{X_{1,n}+\cdots+X_{k_n,n}}$ converge to $\nu_\boxplus.$ So $\omega_n$ norm-converges to $\mathcal{F}{\nu_\boxplus}$. We recall that $\omega_n$ has as right inverse with respect to composition the function $w\mapsto k_nw+(1-k_n)\mathcal{F}_{\nu_{X_{1,n}}}(w)$,
and by Theorem \ref{boolean} and Proposition \ref{continuity} this function converges uniformly on closed bounded sets of the upper half-plane $\mathbb H^+(\ncspace{\cB})$ to $w-h_\nu(w)$. So for $\Im b$ sufficiently large in order for $b-h_\nu(b)$ to belong to $\mathbb H^+(\ncspace{\cB})+i\mathds{1}$,
\begin{eqnarray*}
\lim_{n\to\infty}k_n\cdot h_{\mu_{X_{1,n}}}(b)&=&\lim_{n\to\infty}k_n\cdot h_{\mu_{X_{1,n}}}(\omega_n(k_nb+(1-k_n)\mathcal{F}_{\nu_{X_{1,n}}}(b)))\\
&=&\lim_{n\to\infty}h_{\mu}(b-h_\nu(b))
\end{eqnarray*}
In the last equality we have used the assumption that $k_n\cdot h_{\mu_{X_{1,n}}}\circ\omega_n$ converges to $h_\mu$ uniformly on bounded closed sets of the upper half-plane and that the function $w\mapsto k_nw+(1-k_n)\mathcal{F}_{\nu_{X_{1,n}}}(w)$ converges also uniformly to $w\mapsto w-h_\nu(w)$. Since we assumed $b-h_\nu(b)\in\mathbb H^+(\ncspace{\cB})+i\mathds{1}$, the equality follows. So we have proved (1) must hold. This shows the equivalence between (1) and (2).

The previous remark indicates that indeed the limit of $(\mu_n, \nu_n)^{\cfree\, k_n}$ as $n$ tends to infinity, if existing, must be $\mathcal{BP}(\mu, \nu)$, thus establishing $(2)\Leftrightarrow(3).$

Finally the equivalence of (4) to (1), (2) and (3) is a direct consequence of Lemma \ref{lema19}, the equivalence between norm-convergence on $\Nilp$ and $\mathbb H^+$ for the corresponding transforms, and the definition of the Bercovici-Pata bijection.
\end{proof}

\begin{remark}
We would like to come back once again to the issue of distributions in
$\Sigma_\mathcal{B:D}$, or even in a larger set. As shown by Bercovici
and Voiculescu \cite{BVIUMJ}, one can define free convolutions of probability measures with unbounded support, even when they have no moments whatsoever. In the same paper, they provide an operatorial representation, by showing that if $\mathcal A_1,\mathcal A_2$ are free in $(\mathcal A,\varphi)$, and $X_j=X_j^*$ are {\em unbounded} operators affiliated with $\mathcal A_j$ and having distributions $\mu_{X_j}$, $j=1,2$, then $X_1+X_2$ is a selfadjoint unbounded operator affiliated to $\mathcal A$ and $\mu_{X_1+X_2}=\mu_{X_1}\boxplus\mu_{X_2}.$ In the case of {\em operator-valued} distributions, to our best knowledge such a result does not exist as of now. However, it is very easy to observe that if $(\mathcal A, E_\mathcal B,\mathcal B)$ is an op-valued noncommutative probability space and $X=X^*$ is affiliated to $\mathcal A$, then $(b-X)^{-1}\in\mathcal A$ for any $b\in\mathbb H^+(\mathcal B)$ (one only needs to use analytic functional calculus), so $E_\mathcal B[(b-X)^{-1}]$ is well-defined. However, except when $\mathcal B$ is finite dimensional and $\mathcal A$ is a finite von Neumann algebra, we do not know whether any of the results of Voiculescu and Speicher (existence and good behavior of $R$-transform, subordination etc) remains valid. When $\mathcal B$ is finite dimensional, one can make certain generalizations (for example one can talk of convergence of distributions in $\Sigma_\mathcal{B}^0$ to a distribution in $\Sigma_\mathcal{B}$, or even one without moments - provided we define it appropriately - and even obtain Theorem \ref{hincin} for such a limit when the infinitesimality of an array is defined as convergence of $G_{X_{ij}}$ to $b\mapsto b^{-1}$). However, we repeat that our purpose in this paper was to provide results in maximum generality in terms of $\mathcal B$. We shall postpone a detailed discussion of finite dimensional scalar algebras to future papers.
\end{remark}

\section{The homomorphism property of the Bercovici-Pata bijection}

In this section, we restrict ourselves to the  scalar case, i.e. we consider $\mathcal B=\mathcal D=\mathbb C$ and in Definition \ref{def-cfree} $E_\mathcal B,\theta$ to be a pair of states $\phi,\psi$ . Also, we will use the complex analytic $R$-, $\CR$- and $B$-transforms\footnote{The reader is warned that in scalar probability, the transform $B$ is denoted by $\eta$, unlike in operator-valued probability; we have chosen not to change the notation because we use the operator-valued definitions and results already introduced.}, not their non-commutative versions.

For two pairs of distributions $(\phi_X,\psi_X)$ and $(\phi_Y,\psi_Y)$ given as in Definition \ref{defn1} by the c-free random variables $X$ and $Y$, we shall denote $\psi_X\boxtimes\psi_Y$ for $\phi_{XY}$ and  $(\phi_X,\psi_X)\boxtimes_c(\phi_Y,\psi_Y)$ for $(\phi_{XY},\psi_{XY})$.

In terms of transforms, if
\[
T_\nu(z)=\frac{z}{R^{-1}(z)}, \ \ {}^cT_{\mu,\nu}(z)=\frac{1}{R_{\nu}^{-1}(z)}\CR_{\mu,\nu}(R_{\nu}^{-1}(z))
\]
it has been shown (see \cite{VM}, respectively \cite{cT-transf}) that
\begin{eqnarray*}
&&\hspace{-.5cm}T_{\mu\boxtimes\nu}(z)=T_\mu(z)T_\nu(z) \\
&&\hspace{-.5cm}{}^cT_{(\mu_1,\nu_1)\boxtimes_c(\mu_2,\nu_2)}(z)={}^cT_{(\mu_1,\nu_1)}(z)\cdot {}^cT_{(\mu_2,\nu_2)}(z).
\end{eqnarray*}
(Recall the convention that $f^{-1}(z)$ denotes the compositional inverse of $f$ evaluated in $z$, while $f(z)^{-1}$ means $\frac{1}{f(z)}$.) All these relations hold on a neighborhood of zero.

It has been shown by Belinschi and Nica \cite{BN} that the Boolean Bercovici-Pata bijection is a homomorphism with respect to free multiplicative convolution, meaning that
$$
\mathcal{BP}(\mu\boxtimes\nu)= \mathcal{BP}(\mu)\boxtimes\mathcal{BP}(\nu).
$$
The purpose of this section is to generalize this result to multiplicative c-free convolution. The main tool in \cite{BN} was the observation that $S_{\mathcal{BP}(\mu)}(z)=S_\mu\left(\frac{z}{1-z}\right).$ This observation turns out to be true also for the $T$-transform:
\begin{lemma}
For any compactly supported pair of distributions $(\mu,\nu)$ so that $\int t\,d\nu(t)\neq0$, we have
$$
{}^cT_{\mathcal{BP}(\mu,\nu)}(z)={}^cT_{(\mu,\nu)}\left(\frac{z}{1-z}\right),
$$
for $z$ in a neighborhood of zero.
\end{lemma}
\begin{proof}
 To simplify the notations, if $\sigma$ is a compactly supported real measure, we will write
 \[
 \mathcal M_\sigma(z)=\int \frac{tz}{1-tz}d\sigma(z)=M_\sigma(z)-1.
 \]

%
%
Using equation \eqref{c-free} as re-written in the proof of Lemma \ref{lema19}, we obtain that $\frac{1}{z}B_\mu(z)=\frac{\CR_{(\mu,\nu)}(z+z\mathcal M_\nu(z))}{z+z\cM_\nu(z)}.$ Recall the defining relation for $R_\nu$ as $\cM_\nu(z)=R_\nu(z+z\cM_\nu(z))$; inverting $R_\nu$ to the left gives $R_\nu^{-1}(\cM_\nu(z))=z+z\cM_\nu(z)$, and inverting now $\cM_\nu$ to the right yields
$$
R_\nu^{-1}(z)=\cM_\nu^{-1}(z)(1+z),
$$
again for $z$ in a neighborhood of zero. We compose to the right with $\cM_\nu^{-1}$ (the assumption that $\int t\,d\nu(t)\neq0$ implies that the compositional inverse does exist on a small enough neighborhood of the origin) in the relation that defined $\CR$ and use the above:

\begin{equation}
\frac{1}{\cM_\nu^{-1}(z)}B_\mu(\cM_\nu^{-1}(z))=\frac{\CR_{(\mu,\nu)}(R_\nu^{-1}(z))}{R_\nu^{-1}(z)}={}^cT_{(\mu,\nu)}(z).
\end{equation}
On the other hand, as shown 
in \cite{BN}, $R_{\mathcal{BP}(\nu)}(z)=B_\nu(z)$, so
\begin{eqnarray*}
{}^cT_{\mathcal{BP}(\mu,\nu)}(z) & = & \frac{1}{R_{\mathcal{BP}(\mu)}^{-1}(z)}\CR_{\mathcal{BP}(\mu)}(R_{\mathcal{BP}(\nu)}^{-1}(z))\\
& = & \frac{1}{B_{\nu}^{-1}(z)}\CR_{\mathcal{BP}(\mu)}(B_{\nu}^{-1}(z)).
\end{eqnarray*}
Of course, by using the definition of $B$ we obtain that $B_\nu^{-1}(z)=\cM_\nu^{-1}\left(\frac{z}{1-z}\right)$. Thus, our lemma is proved, since from Definition \ref{def43} we have that $\CR_{\mathcal{BP}(\mu)}=B_\mu$.
\end{proof}

Now the main result of the section follows:
\begin{prop}
For any compactly supported distributions $(\mu_1,\nu_1),(\mu_2,\nu_2)$, we have
$$
\mathcal{BP}((\mu_1,\nu_1)\boxtimes_c(\mu_2,\nu_2))=\mathcal{BP}(\mu_1,\nu_1)\boxtimes_c\mathcal{BP}(\mu_2,\nu_2).
$$
\end{prop}
\begin{proof}
For the second coordinate, this has been proved in \cite{BN}. For the first coordinate, we use the previous lemma to write
\begin{eqnarray*}
{}^cT_{\mathcal{BP}((\mu_1,\nu_1)\boxtimes_c(\mu_2,\nu_2))}(z) & = &
{}^cT_{(\mu_1,\nu_1)\boxtimes_c(\mu_2,\nu_2)}\left(\frac{z}{1-z}\right)\\
& = & {}^cT_{(\mu_1,\nu_1)}\left(\frac{z}{1-z}\right){}^cT_{(\mu_2,\nu_2)}\left(\frac{z}{1-z}\right)\\
& = & {}^cT_{\mathcal{BP}((\mu_1,\nu_1))}(z){}^cT_{\mathcal{BP}((\mu_2,\nu_2))}(z)\\
& = & {}^cT_{\mathcal{BP}(\mu_1,\nu_1)\boxtimes_c\mathcal{BP}(\mu_2,\nu_2)}(z).
\end{eqnarray*}
\end{proof}
We would like to mention, however, that regrettably the c-free convolution of positive probability measures on the positive half-line is not necessarily well-defined, while the Boolean - or c-free -  Bercovici-Pata bijection is not well defined for measures on the unit circle, as sums of unitaries are not unitaries. Thus, the above result must be viewed in terms of algebraic distributions.


\begin{thebibliography}{10}

\bibitem{akhieser} Akhieser, N. I.{\em The classical moment problem
and some related questions in analysis.} Translated by N. Kemmer.
Hafner Publishing Co., New York, 1965.

\bibitem{B-proc}  Belinschi, S. T. {\em C-free convolution of measures with unbounded support}. C${}^\ast$-algebras in Sibiu, 1--7. Theta 2008.

\bibitem{BN} Belinschi, S. T.; Nica, A. {\em On a remarkable semigroup of homomorphisms with respect to free multiplicative convolution}. Indiana Univ. Math. J.

\bibitem{BP} Bercovici, H.; Pata, V. {\em Stable laws and domains of
attraction in free probability theory}, {With an appendix by Philippe Biane},
{Ann. of Math. (2)}, {\bf 149}, {(1999)}, {no.3}, {1023--1060}


\bibitem{BVIUMJ} Bercovici, H.; Voiculescu, D. {\em Free convolutions of measures with unbounded support.}  Indiana Univ. Math. J.  42
(1993),  no. 3, 733--773.

\bibitem{dykemablanchard} Blanchard, E.; Dykema, K. \emph{Embeddings of reduced free products of operator algebras} Pacific Journal of Mathematics 199 (2001), 1-19


\bibitem{boca90} Boca, Florin. \emph{Free products of completely positive maps and spectral sets},  J. Functional Analysis 97 (1991), 251-263


\bibitem{bsk} {Bo\.zejko, Marek; Leinert, Michael; Speicher,
Roland}. {\em Convolution and limit theorems for conditionally
free random variables.} {Pacific J. Math.} {\bf 175}, {(1996)},
no. 2, {357--388}.




\bibitem{isidro} Isidro, Jos\'{e} M.; Stach\'{o}, L\'{a}szl\'{o}. {\em Holomorphic automorphism groups in Banach spaces: an elementary introduction}. North-Holland Mathematics Studies 105, (1985).

    \bibitem{ncfound}  Kaliuzhnyi-Verbovetskyi D.~S. and Vinnikov, Victor. \emph{Foundations of noncommutative function theory},
preprint

\bibitem{ADK} Krysztek, Anna Dorota. {\em Infinite divisibility
for the conditionally free convolution}. IDAQP Vol 10(4) (2007), 499--522.

\bibitem{muraki} Muraki, Naofumi {\em Monotonic independence, monotonic central limit theorem and monotonic law of small numbers}, IDAQP  Vol. 4(1) (2001) pp. 39-58

\bibitem{NSbook} Nica, Alexandru; Speicher, Roland. {\em Lectures on the combinatorics of free probability.} {Cambridge University Press}, {2006}.

\bibitem{paulsen} Paulsen, Vern. {\em Completely bounded maps and operator algebras.} Cambridge Studies in Advanced Mathematics 78, Cambridge University Press, Cambridge, UK, 2002.

\bibitem{P} Popa, Mihai. \emph{A combinatorial approach to monotonic independence over a C*-algebra}, Pac J. of Math, Vol. 237 (2008), No. 2, 299-325

\bibitem{mv-bool}Popa, Mihai. \emph{A new proof for the multiplicative property of the boolean cumulants with applications
to the operator-valued case} Colloquium Mathematicum, Vol 117 (2009), No. 1 , 81-93


\bibitem{mvcomstoc}Popa, Mihai. \emph{Multilinear function series in conditionally free probability with amalgamation}
Com. on Stochastic Anal., Vol 2, No 2 (Aug 2008)

\bibitem{popa-vinnikov}  Popa, Mihai; Vinnikov, Victor. {\em Non-Commutative Functions and Non-Commutative Free Levy-Hincin Formula}, preprint

\bibitem{cT-transf} Popa, Mihai; Wang, Jiun-Chau. {\em On multiplicative conditionally free convolution}, Trans. Amer. Math. Soc., to appear; available also on arXiv:0805.0257

\bibitem{speicherhab} Speicher, Roland. {\em Combinatorial theory of the free product with amalgamation and operator-valued free probability theory.} {Memoirs of the Amer. Math. Soc.} {132} {(1998)}, x+88

\bibitem{SW} Speicher, Roland; Woroudi, Reza. {\em Boolean
convolution.} Fields Institute Communications, Vol. 12 (D.
Voiculescu, ed), AMS, 1997, 267--279.

\bibitem{taylor} J.L.~Taylor, \emph{Functions of several noncommuting variables}, Bull. Amer. Math. Soc. \textbf{79}
(1973), 1--34.

\bibitem{V-JFA} Voiculescu, D. V. {\em Addition of certain
non-commutative random variables}, {J. Funct. Anal.} {\bf 66} {(1986)}, {323--346}.

\bibitem{VM} {D. Voiculescu},
{\em Multiplication of certain noncommuting random variables},
J. Operator Theory {\bf 18}(1987), 223--235.

\bibitem{V*} Voiculescu, D. V. {\em Operations on certain
non-commutative operator-valued random variables}, Ast\'erisque
(1995), no. 232, 243--275.

\bibitem{V2} Voiculescu, D. V. {\em The coalgebra of the free
difference quotient and free probability.}
Internat. Math. Res. Not. (2000) No. 2.

\bibitem{V1} Voiculescu, D. V. {\em Free Analysis Questions I: Duality Transform for the Coalgebra of $\partial_{X:B}$}
Internat. Math. Res. Not. (2004), No. 16.




\bibitem{W} Wang, J.-C. {\em Limit theorems for additive c-free convolution}, preprint arXiv:0805.0607

\end{thebibliography}
\end{document}